\documentclass{amsart}
\usepackage[T1]{fontenc}
\usepackage[utf8]{inputenc}
\usepackage{lmodern}
\usepackage{amssymb}
\usepackage{tikz-cd}
\numberwithin{equation}{section}

\newtheorem{thm}{Theorem}[section]
\newtheorem{prop}[thm]{Proposition}
\newtheorem{lem}[thm]{Lemma}
\newtheorem{cor}[thm]{Corollary}

\theoremstyle{definition}
\newtheorem{defn}[thm]{Definition}

\theoremstyle{remark}
\newtheorem{rem}[thm]{Remark}
\newtheorem*{claim}{Claim}

\newcommand{\Z}{\mathbb{Z}}

\newcommand{\R}{\mathbb{R}}

\DeclareMathOperator{\Hom}{Hom}

\DeclareMathOperator{\Ker}{Ker}
\DeclareMathOperator{\Coker}{Cok}
\DeclareMathOperator{\Ima}{Im}
\newcommand{\trivrep}{\mathbf{1}}
\newcommand{\aff}{\mathrm{aff}}
\newcommand{\St}{\mathrm{St}}
\DeclareMathOperator{\Stab}{Stab}

\newcommand*{\opposite}[1]{#1^{\mathrm{op}}}

\DeclareMathOperator{\val}{val}

\newcommand{\Refs}[1]{\mathrm{Ref}(#1)}
\newcommand{\coroot}[1]{#1^\vee}

\title{Involutions on pro-$p$-Iwahori Hecke algebras}
\author{Noriyuki Abe}
\address{Department of Mathematics, Hokkaido University, Kita 10, Nishi 8, Kita-Ku, Sapporo, Hokkaido, 060-0810, Japan}
\email{abenori@math.sci.hokudai.ac.jp}
\subjclass[2010]{20C08, 20G25}
\begin{document}
\begin{abstract}
The pro-$p$-Iwahori Hecke algebra has an involution $\iota$ defined in terms of Iwahori-Matsumoto basis.
Then for a module $\pi$ of pro-$p$-Iwahori Hecke, $\pi^\iota = \pi\circ \iota$ is also a module.
We calculate $\pi^\iota$ for simple modules $\pi$.
We also calculate the dual of $\pi$.
\end{abstract}
\maketitle
\section{Introduction}\label{sec:Introduction}
This is the sequel of \cite{arXiv:1612.01312} and the aim of these papers are to calculate the extension of simple modules of pro-$p$-Iwahori Hecke algebras.
The calculation will be appeared in sequel and we will use the results of this paper.

Let $G$ be a connected reductive group over a non-archimedean local field with residue characteristic $p$.
For a field $C$, we can attache the pro-$p$-Iwahori Hecke algebra of $G$.
This is the convolution algebra of compactly supported functions which is bi-invariant under the pro-$p$ radical of an Iwahori subgroup.
If the characteristic of $C$ is $p$, then this algebra plays an important role for the representation theory of $G$ over $C$ (cf.~\cite{MR3600042}).

The main object of this paper are the anti-involution $\zeta$ and the involution $\iota$ when the characteristic of $C$ is $p$.
These are defined as follows.
\begin{itemize}
\item $\zeta$: Let $W(1)$ be the ``pro-$p$ Weyl group'' (see subsection~\ref{subsec:The data from a group} for the precise definition).
Then $\mathcal{H}$ has a basis $\{T_w\mid w\in W(1)\}$ parametrized by $W(1)$ which is called \emph{Iwahori-Matsumoto basis}.
The anti-involution $\zeta$ is defined by $\zeta(T_w) = T_{w^{-1}}$.
\item $\iota$: We also have another basis of $\mathcal{H}$ defnoted by $\{T_w^*\mid w\in W(1)\}$.
Then the involution $\iota$ is defined by $\iota(T_w) = (-1)^{\ell(w)}T_w^*$ where $\ell$ is the length function on $W(1)$.
\end{itemize}
By the multiplication rule of $\mathcal{H}$ in terms of the basis $\{T_w\mid w\in W(1)\}$ (the braid relations and the quadratic relations), these maps respect the multiplication.

Let $\pi$ be a right $\mathcal{H}$-module.
Then we can attache the following two modules.
\begin{itemize}
\item $\pi^* = \Hom_C(\pi,C)$ where the action of $X\in \mathcal{H}$ on $f\in \pi^*$ is given by $(fX)(v) = f(v\zeta(X))$ for $v\in \pi$.
\item $\pi^\iota = \pi\circ\iota$.
\end{itemize}
If $\pi$ is simple then $\pi^*$ and $\pi^\iota$ are also simple.
We determine these modules (Theorem~\ref{thm:twist of simple module}, \ref{thm:dual of simples}).

\subsection{Organization of this paper}
In the next section we give notation and recall some results on pro-$p$-Iwahori Hecke algebras.
We use the same notation as \cite{arXiv:1612.01312} and often refer this paper.
In Section~\ref{sec:Twist}, we calculate $\pi^\iota$ for simple modules $\pi$.
Since the construction of simple modules is divided into three steps, the calculation also has three steps, namely when $\pi$ is the parabolic induction, generalized Steinberg modules and supersingular modules.
The hardest step is for the generalized Steinberg modules which will be calculated from subsection~\ref{subsec:twist:Steinberg modules} to \ref{subsec:exactness}.
The calculation of the dual of simple modules will be done in Section~\ref{sec:Dual}.
In the calculation we use results in Section~\ref{sec:Twist}.

\subsection*{Acknowledgments}
Most of this work was done during my pleasant stay at Institut de mathématiques de Jussieu. The work is supported by JSPS KAKENHI Grant Number 26707001.

\section{Preliminaries}\label{sec:Preliminaries}
\subsection{Pro-$p$-Iwahori Hecke algebra}\label{subsec:Prop-p-Iwahori Hecke algebra}
Let $\mathcal{H}$ be a pro-$p$-Iwahori Hecke algebra over a commutative ring $C$~\cite{MR3484112}.
We study modules over $\mathcal{H}$ in this paper.
\emph{In this paper, a module means a right module.}
The algebra $\mathcal{H}$ is defined with a combinatorial data $(W_\aff,S_\aff,\Omega,W,W(1),Z_\kappa)$ and a parameter $(q,c)$.

We recall the definitions.
The data satisfy the following.
\begin{itemize}
\item $(W_\aff,S_\aff)$ is a Coxeter system.
\item $\Omega$ acts on $(W_\aff,S_\aff)$.
\item $W = W_\aff\rtimes \Omega$.
\item $Z_\kappa$ is a finite commutative group.
\item The group $W(1)$ is an extension of $W$ by $Z_\kappa$, namely we have an exact sequence $1\to Z_\kappa\to W(1)\to W\to 1$.
\end{itemize}
The subgroup $Z_\kappa$ is normal in $W(1)$.
Hence the conjugate action of $w\in W(1)$ induces an automorphism of $Z_\kappa$, hence of the group ring $C[Z_\kappa]$.
We denote it by $c\mapsto w\cdot c$.

Let $\Refs{W_\aff}$ be the set of reflections in $W_\aff$ and $\Refs{W_\aff(1)}$ the inverse image of $\Refs{W_\aff}$ in $W(1)$.
The parameter $(q,c)$ is maps $q\colon S_\aff\to C$ and $c\colon \Refs{W_\aff(1)}\to C[Z_\kappa]$ with the following conditions. (Here the image of $s$ by $q$ (resp.\ $c$) is denoted by $q_s$ (resp.\ $c_s$).)
\begin{itemize}
\item For $w\in W$ and $s\in S_\aff$, if $wsw^{-1}\in S_\aff$ then $q_{wsw^{-1}} = q_s$.
\item For $w\in W(1)$ and $s\in \Refs{W_\aff(1)}$, $c_{wsw^{-1}} = w\cdot c_s$.
\item For $s\in \Refs{W_\aff(1)}$ and $t\in Z_\kappa$, we have $c_{ts} = tc_s$.
\end{itemize}
Let $S_\aff(1)$ be the inverse image of $S_\aff$ in $W(1)$.
For $s\in S_\aff(1)$, we write $q_s$ for $q_{\bar{s}}$ where $\bar{s}\in S_\aff$ is the image of $s$.
The length function on $W_\aff$ is denoted by $\ell$ and its inflation to $W$ and $W(1)$ is also denoted by $\ell$.

The $C$-algebra $\mathcal{H}$ is a free $C$-module and has a basis $\{T_w\}_{w\in W(1)}$.
The multiplication is given by
\begin{itemize}
\item (Quadratic relations) $T_{s}^2 = q_sT_{s^2} + c_sT_s$ for $s\in S_\aff(1)$.
\item (Braid relations) $T_{vw} = T_vT_w$ if $\ell(vw) = \ell(v) + \ell(w)$.
\end{itemize}
We extend $q\colon S_\aff\to C$ to $q\colon W\to C$ as follows.
For $w\in W$, take $s_1,\dots,s_l\in S_\aff$ and $u\in\Omega$ such that $w = s_1\dotsm s_l u$.
Then put $q_w = q_{s_1}\dotsm q_{s_l}$.
From the definition, we have $q_{w^{-1}} = q_w$.
We also put $q_w = q_{\overline{w}}$ for $w\in W(1)$ with the image $\overline{w}$ in $W$.

\subsection{The data from a group}\label{subsec:The data from a group}
Let $F$ be a non-archimedean local field, $\kappa$ its residue field, $p$ its residue characteristic and $G$ a connected reductive group over $F$.
We can get the data in the previous subsection from $G$ as follows.
See \cite{MR3484112}, especially 3.9 and 4.2 for the details.

Fix a maximal split torus $S$ and denote the centralizer of $S$ by $Z$.
Let $Z^0$ be the unique parahoric subgroup of $Z$ and $Z(1)$ its pro-$p$ radical.
Then the group $W(1)$ (resp.\ $W$) is defined by $W(1) = N_G(Z)/Z(1)$ (resp.\ $W = N_G(Z)/Z^0$) where $N_G(Z)$ is the normalizer of $Z$ in $G$.
We also have $Z_\kappa = Z^0/Z(1)$.
Let $G'$ be the group generated by the unipotent radical of parabolic subgroups \cite[II.1]{MR3600042} and $W_\aff$ the image of $G'\cap N_G(Z)$ in $W$.
Then this is a Coxeter group.
Fix a set of simple reflections $S_\aff$.
The group $W$ has the natural length function and let $\Omega$ be the set of length zero elements in $W$.
Then we get the data $(W_\aff,S_\aff,\Omega,W,W(1),Z_\kappa)$.

Consider the apartment attached to $S$ and an alcove surrounded by the hyperplanes fixed by $S_\aff$.
Let $I(1)$ be the pro-$p$-Iwahori subgroup attached to this alcove.
Then with $q_s = \#(I(1)\widetilde{s}I(1)/I(1))$ for $s\in S_\aff$ with a lift $\widetilde{s}\in N_G(Z)$ and suitable $c_s$, the algebra $\mathcal{H}$ is isomorphic to the Hecke algebra attached to $(G,I(1))$ \cite[Proposition~4.4]{MR3484112}.

In this paper, \emph{the data $(W_\aff,S_\aff,\Omega,W,W(1),Z_\kappa)$ and the parameters $(q,c)$ come from $G$ in this way.}
Let $W_\aff(1)$ be the image of $G'\cap N_G(Z)$ in $W(1)$ and put $\mathcal{H}_\aff = \bigoplus_{w\in W_\aff(1)}CT_w$.
This is a subalgebra of $\mathcal{H}$.

\subsection{The algebra $\mathcal{H}[q_s]$ and $\mathcal{H}[q_s^{\pm 1}]$}\label{subsec:The algebra H[q_s] and H[q_s^pm]}
For each $s\in S_\aff$, let $\mathbf{q}_s$ be an indeterminate such that if $wsw^{-1}\in S_\aff$ for $w\in W$, we have $\mathbf{q}_{wsw^{-1}} = \mathbf{q}_s$.
Let $C[\mathbf{q}_s]$ be a polynomial ring with these indeterminate.
Then with the parameter $s\mapsto \mathbf{q}_s$ and the other data coming from $G$, we have the algebra.
This algebra is denoted by $\mathcal{H}[\mathbf{q}_s]$ and we put $\mathcal{H}[\mathbf{q}_s^{\pm 1}] = \mathcal{H}[\mathbf{q}_s]\otimes_{C[\mathbf{q}_s]}C[\mathbf{q}_s^{\pm 1}]$.
Under $\mathbf{q}_s\mapsto \#(I(1)\widetilde{s}I(1)/I(1))\in C$ where $\widetilde{s}\in N_G(Z)$ is a lift of $s$, we have $\mathcal{H}[\mathbf{q}_s]\otimes_{C[\mathbf{q}_s]}C \simeq \mathcal{H}$.
As an abbreviation, we denote $\mathbf{q}_s$ by just $q_s$.
Consequently we denote by $\mathcal{H}[q_s]$ (resp.\ $\mathcal{H}[q_s^{\pm 1}]$).

Since $q_s$ is invertible in $\mathcal{H}[q_s^{\pm 1}]$, we can do some calculations in $\mathcal{H}[q_s^{\pm 1}]$ with $q_s^{-1}$.
If the result can be stated in $\mathcal{H}[q_s]$, then this is an equality in $\mathcal{H}[q_s]$ since $\mathcal{H}[q_s]$ is a subalgebra of $\mathcal{H}[q_s^{\pm 1}]$ and by specializing, we can get some equality in $\mathcal{H}$.
See \cite[4.5]{MR3484112} for more details.

\subsection{The root system and the Weyl groups}
Let $W_0 = N_G(Z)/Z$ be the finite Weyl group.
Then this is a quotient of $W$.
Recall that we have the alcove defining $I(1)$.
Fix a special point $\boldsymbol{x}_0$ from the border of this alcove.
Then $W_0\simeq \Stab_W\boldsymbol{x}_0$ and the inclusion $\Stab_W\boldsymbol{x}_0\hookrightarrow W$ is a splitting of the canonical projection $W\to W_0$.
Throughout this paper, we fix this special point and regard $W_0$ as a subgroup of $W$.
Set $S_0 = S_\aff\cap W_0\subset W$.
This is a set of simple reflections in $W_0$.
For each $w\in W_0$, we fix a representative $n_w\in W(1)$ such that $n_{w_1w_2} = n_{w_1}n_{w_2}$ if $\ell(w_1w_2) = \ell(w_1) + \ell(w_2)$.

The group $W_0$ is the Weyl group of the root system $\Sigma$ attached to $(G,S)$.
Our fixed alcove and special point give a positive system of $\Sigma$, denoted by $\Sigma^+$.
The set of simple roots is denoted by $\Delta$.
As usual, for $\alpha\in\Delta$, let $s_\alpha\in S_0$ be a simple reflection for $\alpha$.

The kernel of $W(1)\to W_0$ (resp.\ $W\to W_0$) is denoted by $\Lambda(1)$ (resp.\ $\Lambda$).
Then $Z_\kappa\subset \Lambda(1)$ and we have $\Lambda = \Lambda(1)/Z_\kappa$.
The group $\Lambda$ (resp.\ $\Lambda(1)$) is isomorphic to $Z/Z^0$ (resp.\ $Z/Z(1)$).
Any element in $W(1)$ can be uniquely written as $n_w \lambda$ where $w\in W_0$ and $\lambda\in \Lambda(1)$.
We have $W = W_0\ltimes \Lambda$.

\subsection{The map $\nu$}
The group $W$ acts on the apartment attached to $S$ and the action of $\Lambda$ is by the translation.
Since the group of translations of the apartment is $X_*(S)\otimes_{\Z}\R$, we have a group homomorphism $\nu\colon \Lambda\to X_*(S)\otimes_{\Z}\R$.
The compositions $\Lambda(1)\to \Lambda\to X_*(S)\otimes_{\Z}\R$ and $Z\to \Lambda\to X_*(S)\otimes_{\Z}\R$ are also denoted by $\nu$.
The homomorphism $\nu\colon Z\to X_*(S)\otimes_{\Z}\R\simeq \Hom_\Z(X^*(S),\R)$ is characterized by the following: For $t\in S$ and $\chi\in X^*(S)$, we have $\nu(t)(\chi) = -\val(\chi(t))$ where $\val$ is the normalized valuation of $F$.
The kernel of $\nu\colon Z\to X_*(S)\otimes_{\Z}\R$ is equal to the maximal compact subgroup $\widetilde{Z}$ of $Z$.
In particular, $\Ker(\Lambda(1)\xrightarrow{\nu} X_*(S)\otimes_{\Z}\R) = \widetilde{Z}/Z(1)$ is a finite group.

We call $\lambda\in \Lambda(1)$ dominant (resp.\ anti-dominant) if $\nu(\lambda)$ is dominant (resp.\ anti-dominant).

Since the group $W_\aff$ is a Coxeter system, it has the Bruhat order denoted by $\le$.
For $w_1,w_2\in W_\aff$, we write $w_1 < w_2$ if there exists $u\in \Omega$ such that $w_1u,w_2u\in W_\aff$ and $w_1u < w_2u$.
Moreover, for $w_1,w_2\in W(1)$, we write $w_1 < w_2$ if $w_1\in W_{\aff}(1)w_2$ and $\overline{w}_1 < \overline{w}_2$ where $\overline{w}_1,\overline{w}_2$ are the image of $w_1,w_2$ in $W$, respectively.
We write $w_1\le w_2$ if $w_1 < w_2$ or $w_1 = w_2$.

\subsection{Other basis}
For $w\in W(1)$, take $s_1,\dotsm,s_l\in S_\aff(1)$ and $u\in W(1)$ such that $l = \ell(w)$, $\ell(u) = 0$ and $w = s_1\dotsm s_l u$.
Set $T_w^* = (T_{s_1} - c_{s_1})\dotsm (T_{s_l} - c_{s_l})T_u$.
Then this does not depend on the choice and we have $T_w^* \in T_w + \sum_{v <  w}C T_v$.
In particular, $\{T_w^*\}_{w\in W(1)}$ is a basis of $\mathcal{H}$.
In $\mathcal{H}[q_s^{\pm 1}]$, we have $T_w^* = q_wT_{w^{-1}}^{-1}$.

For simplicity, we always assume that our commutative ring $C$ contains a square root of $q_s$ which is denoted by $q_s^{1/2}$ for $s\in S_\aff$.
For $w = s_1\dotsm s_lu$ where $\ell(w) = l$ and $\ell(u) = 0$, $q_w^{1/2} = q_{s_1}^{1/2}\dotsm q_{s_l}^{1/2}$ is a square root of $q_w$.
For a spherical orientation $o$, there is a basis $\{E_o(w)\}_{w\in W(1)}$ of $\mathcal{H}$ introduced in \cite[5]{MR3484112}.
We have
\[
E_o(w)\in T_w + \sum_{v < w}CT_v.
\]
This satisfies the following product formula~\cite[Theorem~5.25]{MR3484112}.
\begin{equation}\label{eq:product formula}
E_o(w_1)E_{o\cdot w_1}(w_2) = q_{w_1w_2}^{-1/2}q_{w_1}^{1/2}q_{w_2}^{1/2}E_o(w_1w_2).
\end{equation}
\begin{rem}
Since we do not assume that $q_s$ is invertible in $C$, $q_{w_1w_2}^{-1/2}q_{w_1}^{1/2}q_{w_2}^{1/2}$ does not make sense in a usual way.
See \cite[Remark~2.2]{arXiv:1612.01312}.
\end{rem}

\subsection{Levi subalgebra}
Since we have a positive system $\Sigma^+$, we have a minimal parabolic subgroup $B$ with a Levi part $Z$.
In this paper, \emph{parabolic subgroups are always standard, namely containing $B$}.
Note that such parabolic subgroups correspond to subsets of $\Delta$.

Let $P$ be a parabolic subgroup.
Attached to the Levi part of $P$ containing $Z$, we have the data $(W_{\aff,P},S_{\aff,P},\Omega_P,W_P,W_P(1),Z_\kappa)$ and the parameters $(q_P,c_P)$.
Hence we have the algebra $\mathcal{H}$.
The parameter $c_P$ is given by the restriction of $c$, hence we denote it just by $c$.
The parameter $q_P$ is defined as in \cite[4.1]{arXiv:1406.1003_accepted}.

For the objects attached to this data, we add the suffix $P$.
We have the set of simple roots $\Delta_P$, the root system $\Sigma_P$ and its positive system $\Sigma_P^+$, the finite Weyl group $W_{0,P}$, the set of simple reflections $S_{0,P}\subset W_{0,P}$, the length function $\ell_P$ and the base $\{T^P_w\}_{w\in W_P(1)}$, $\{T^{P*}_w\}_{w\in W_P(1)}$ and $\{E^P_o(w)\}_{w\in W_P(1)}$ of $\mathcal{H}_P$.
Note that we have no $\Lambda_P$, $\Lambda_P(1)$ and $Z_{\kappa,P}$ since they are equal to $\Lambda$, $\Lambda(1)$ and $Z_\kappa$.

An element $w = n_v\lambda\in W_P(1)$ where $v\in W_{0,P}$ and $\lambda\in\Lambda(1)$ is called $P$-positive (resp.\ $P$-negative) if $\langle \alpha,\nu(\lambda)\rangle\le 0$ (resp.\ $\langle \alpha,\nu(\lambda)\rangle\ge 0$) for any $\alpha\in\Sigma^+\setminus\Sigma_P^+$.
Let $W_P^+(1)$ (resp.\ $W_P^-(1)$) be the set of $P$-positive (resp.\ $P$-negative) elements and put $\mathcal{H}_P^{\pm} = \bigoplus_{w\in W_P^{\pm}(1)}CT_w^P$.
These are subalgebras of $\mathcal{H}_P$ \cite[Lemma~4.1]{arXiv:1406.1003_accepted}.

\begin{prop}[{\cite[Theorem~1.4]{MR3437789}}]\label{prop:localization as Levi subalgebra}
Let $\lambda_P^+$ (resp.\ $\lambda_P^-$) be in the center of $W_P(1)$ such that $\langle \alpha,\nu(\lambda_P^+)\rangle < 0$ (resp.\ $\langle \alpha,\nu(\lambda_P^-)\rangle > 0$) for all $\alpha\in\Sigma^+\setminus\Sigma_P^+$.
Then $T^P_{\lambda_P^+} = T^{P*}_{\lambda_P^+} = E^P_{o_{-,P}}(\lambda_P^+)$ (resp.\ $T^P_{\lambda_P^-} = T^{P*}_{\lambda_P^-} = E^P_{o_{-,P}}(\lambda_P^-)$) is in the center of $\mathcal{H}_P$ and we have $\mathcal{H}_P = \mathcal{H}_P^+E^P_{o_{-,P}}(\lambda_P^+)^{-1}$ (resp.\ $\mathcal{H}_P = \mathcal{H}_P^-E^P_{o_{-,P}}(\lambda_P^-)^{-1}$).
\end{prop}
Note that such $\lambda_P^\pm$ always exists \cite[Lemma~2.4]{arXiv:1612.01312}.

We define $j_P^{\pm}\colon \mathcal{H}_P^\pm\to \mathcal{H}$ and $j_P^{\pm *}\colon \mathcal{H}_P^{\pm}\to \mathcal{H}$ by $j_P^{\pm}(T_w^P) = T_w$ and $j_P^{\pm *}(T_w^{P*}) = T_w^*$ for $w\in W_P^\pm(1)$.
Then these are algebra homomorphisms.

Let $Q$ be a parabolic subgroup containing $P$ and let $W_P^{Q+}(1)$ (resp.\ $W_P^{Q-}(1)$) be the set of $n_w\lambda$ where $\langle \alpha,\nu(\lambda)\rangle \le 0$ (resp.\ $\langle \alpha,\nu(\lambda)\rangle \ge 0$) for any $\alpha\in\Sigma_Q^+\setminus\Sigma_P^+$ and $w\in W_{0,P}$.
Put $\mathcal{H}_P^{Q\pm } = \bigoplus_{w\in W_P^{Q\pm}(1)}C T_w^P\subset \mathcal{H}_P$.
Then we have homomorphisms $j_P^{Q\pm },j_P^{Q\pm *}\colon \mathcal{H}_P^{Q\pm}\to \mathcal{H}_Q$ defined by a similar way.

\subsection{Parabolic induction}\label{subsec:Preliminaries, Parabolic induction}
Let $P$ be the parabolic subgroup and $\sigma$ an $\mathcal{H}_P$-module. (This is a right module as in subsection~\ref{subsec:Prop-p-Iwahori Hecke algebra}.)
Then we define an $\mathcal{H}$-module $I_P(\sigma)$ by
\[
I_P(\sigma) = \Hom_{(\mathcal{H}_P^-,j_P^{-*})}(\mathcal{H},\sigma).
\]
We call $I_P$ the parabolic induction.
For $P\subset P_1$, we write
\[
I_P^{P_1}(\sigma) = \Hom_{(\mathcal{H}_P^{P_1-},j_P^{P_1-*})}(\mathcal{H}_{P_1},\sigma).
\]

Let $P$ be a parabolic subgroup.
Set $W_0^P = \{w\in W_0\mid w(\Delta_P) \subset\Sigma^+\}$.
Then the multiplication map $W_0^P\times W_{0,P}\to W_0$ is bijective and for $w_1\in W_0^P$ and $w_2\in W_{0,P}$, we have $\ell(w_1w_2) = \ell(w_1) + \ell(w_2)$.
We also put ${}^PW_0 = \{w\in W_0\mid w^{-1}(\Delta_P) \subset\Sigma^+\}$.
Then the multiplication map $W_{0,P}\times {}^PW_0 \to W_0$ is bijective and for $w_1\in W_{0,P}$ and $w_2\in {}^PW_0$, we have $\ell(w_1w_2) = \ell(w_1) + \ell(w_2)$.
See \cite[Proposition~2.9]{arXiv:1612.01312} for the following proposition.
\begin{prop}\label{prop:decomposition of I_P}
Let $P$ be a parabolic subgroup and $\sigma$ an $\mathcal{H}_P$-module.
\begin{enumerate}
\item The map $I_P(\sigma)\ni\varphi\mapsto (\varphi(T_{n_w}))_{w\in W_0^P}\in \bigoplus_{w\in W_0^P}\sigma$ is bijective.
\item The map $I_P(\sigma)\ni\varphi\mapsto (\varphi(T_{n_w}^*))_{w\in W_0^P}\in \bigoplus_{w\in W_0^P}\sigma$ is bijective.
\end{enumerate}
\end{prop}

\begin{prop}[{\cite[Proposition~4.12]{arXiv:1406.1003_accepted}}]\label{prop:I_P as A-module}
Assume that $q_s = 0$ for any $s\in S_\aff$.
Let $w\in W_0^P$ and $\lambda\in \Lambda(1)$.
Then for $\varphi\in I_P(\sigma)$, we have 
\[
(\varphi E_{o_-}(\lambda))(T_{n_w}) =
\begin{cases}
\varphi(T_{n_w})\sigma(E^P_{o_{-,P}}(n_w^{-1}\cdot \lambda)) & (n_w^{-1}\cdot \lambda\in W_P^-(1)),\\
0 & (n_w^{-1}\cdot \lambda\notin W_P^-(1)).
\end{cases}
\]
\end{prop}

We also define
\[
I_P'(\sigma) = \Hom_{(\mathcal{H}_P^-,j_P^-)}(\mathcal{H},\sigma)
\]
and $I_P^{P_1\prime}$ by the similar way.

\subsection{Twist by $n_{w_Gw_P}$}
For a parabolic subgroup $P$, let $w_P$ be the longest element in $W_{0,P}$.
In particular, $w_G$ is the longest element in $W_0$.
Let $P'$ be a parabolic subgroup corresponding to $-w_G(\Delta_P)$, in other words, $P' = n_{w_Gw_P}\opposite{P}n_{w_Gw_P}^{-1}$ where $\opposite{P}$ is the opposite parabolic subgroup of $P$ with respect to the Levi part of $P$ containing $Z$.
Set $n = n_{w_Gw_P}$.
Then the map $\opposite{P}\to P'$ defined by $p\mapsto npn^{-1}$ is an isomorphism which preserves the data used to define the pro-$p$-Iwahori Hecke algebras.
Hence $T^P_w\mapsto T^{P'}_{nwn^{-1}}$ gives an isomorphism $\mathcal{H}_P\to \mathcal{H}_{P'}$.
This sends $T_w^{P*}$ to $T_{nwn^{-1}}^{P'*}$ and $E^P_{o_{+,P}\cdot v}(w)$ to $E^{P'}_{o_{+,P'}\cdot nvn^{-1}}(nwn^{-1})$ where $v\in W_{0,P}$.

Let $\sigma$ be an $\mathcal{H}_P$-module.
Then we define an $\mathcal{H}_{P'}$-module $n_{w_Gw_P}\sigma$ via the pull-back of the above isomorphism.
Namely, for $w\in W_{P'}(1)$, we put $(n_{w_Gw_P}\sigma)(T^{P'}_w) = \sigma(T^P_{n_{w_Gw_P}^{-1}wn_{w_Gw_P}})$.

\subsection{The extension and the generalized Steinberg modiles}
Let $P$ be a parabolic subgroup and $\sigma$ an $\mathcal{H}_P$-module.
For $\alpha\in\Delta$, let $P_\alpha$ be a parabolic subgroup corresponding to $\Delta_P\cup \{\alpha\}$.
Then we define $\Delta(\sigma)\subset\Delta$ by
\begin{align*}
&\Delta(\sigma)\\& = \{\alpha\in\Delta\mid \langle \Delta_P,\coroot{\alpha}\rangle = 0,\ \text{$\sigma(T^P_\lambda) = 1$ for any $\lambda\in W_{\aff,P_\alpha}(1)\cap \Lambda(1)$}\}\cup \Delta_P.
\end{align*}
Let $P(\sigma)$ be a parabolic subgroup corresponding to $\Delta(\sigma)$.
\begin{prop}[{\cite[Corollary~3.9]{arXiv:1703.10384}}]
Let $\sigma$ be an $\mathcal{H}_P$-module and $Q$ a parabolic subgroup between $P$ and $P(\sigma)$.
Denote the parabolic subgroup corresponding to $\Delta_Q\setminus\Delta_P$ by $P_2$.
Then there exist a unique $\mathcal{H}_Q$-module $e_Q(\sigma)$ acting on the same space as $\sigma$ such that
\begin{itemize}
\item $e_Q(\sigma)(T_w^{Q*}) = \sigma(T_w^{P*})$ for any $w\in W_P(1)$.
\item $e_Q(\sigma)(T_w^{Q*}) = 1$ for any $w\in W_{P_2,\aff}(1)$.
\end{itemize}
Moreover, one of the following condition gives a characterization of $e_Q(\sigma)$.
\begin{enumerate}
\item For any $w\in W_P^{Q-}(1)$, $e_Q(\sigma)(T_w^{Q*}) = \sigma(T_w^{P*})$ (namely, $e_Q(\sigma) \simeq \sigma$ as $(\mathcal{H}_P^{Q-},j_P^{Q-*})$-modules) and for any $w\in W_{\aff,P_2}(1)$, $e_Q(\sigma)(T_w^{Q*}) = 1$.
\item For any $w\in W_P^{Q+}(1)$, $e_Q(\sigma)(T_w^{Q*}) = \sigma(T_w^{P*})$ and for any $w\in W_{\aff,P_2}(1)$, $e_Q(\sigma)(T_w^{Q*}) = 1$.
\end{enumerate}
\end{prop}

We call $e_Q(\sigma)$ the extension of $\sigma$ to $\mathcal{H}_Q$.
A typical example of the extension is the trivial representation $\trivrep = \trivrep_G$.
This is a one-dimensional $\mathcal{H}$-module defined by $\trivrep(T_w) = q_w$, or equivalently $\trivrep(T_w^*) = 1$.
We have $\Delta(\trivrep_P) = \{\alpha\in\Delta\mid \langle \Delta_P,\coroot{\alpha}\rangle = 0\}\cup \Delta_P$ and, if $Q$ is a parabolic subgroup between $P$ and $P(\trivrep_P)$, we have $e_Q(\trivrep_P) = \trivrep_Q$

Let $P(\sigma)\supset P_0\supset Q_1\supset Q\supset P$.
Then as in \cite[4.5]{arXiv:1406.1003_accepted}, we have $I^{P_0}_{Q_1}(e_{Q_1}(\sigma))\subset I^{P_0}_Q(e_Q(\sigma))$.
Define
\[
\St_Q^{P_0}(\sigma) = \Coker\left(\bigoplus_{Q_1\supsetneq Q}I_{Q_1}^{P_0}(e_{Q_1}(\sigma))\to I^{P_0}_Q(e_Q(\sigma))\right).
\]
When $P_0 = G$, we write $\St_Q(\sigma)$.

In the rest of this subsection, we assume that $P(\sigma) = G$.
As we mentioned in the above, for $Q_1\supset Q\supset P$, we have $I_{Q_1}(e_{Q_1}(\sigma))\hookrightarrow I_Q(e_Q(\sigma))$.
The proof of \cite[Lemma~4.23]{arXiv:1406.1003_accepted} implies the following lemma.
\begin{lem}\label{lem:embedding of smaller parabolic induction}
Assume that $q_s = 0$ for any $s\in S_\aff$.
The diagram
\[
\begin{tikzcd}
I_{Q_1}(e_{Q_1}(\sigma)) \arrow[hookrightarrow]{r}\arrow[-]{d}{\wr} & I_Q(e_Q(\sigma))\arrow[-]{d}{\wr}\\
\bigoplus_{w\in W_0^{Q_1}}\sigma \arrow[hookrightarrow]{r} & \bigoplus_{w\in W_0^{Q}}\sigma
\end{tikzcd}
\]
is commutative.
Here the embedding $\bigoplus_{w\in W_0^{Q_1}}\sigma\hookrightarrow \bigoplus_{w\in W_0^{Q_1}}\sigma$ is induced by $W_0^{Q_1}\hookrightarrow W_0^Q$.
\end{lem}

\begin{lem}\label{lem:I_P+extension as A-module}
Assume that $q_s = 0$ for any $s\in S_\aff$.
Let $\varphi\in I_Q(e_Q(\sigma))$.
Then for $w\in W_0^Q$ and $\lambda\in\Lambda(1)$, we have
\[
(\varphi E_{o_-}(\lambda))(T_{n_w})
=
\begin{cases}
\varphi(T_{n_w})\sigma(E^P_{o_{-,P}}(n_w^{-1}\cdot \lambda)) & (\text{$n_w^{-1}\cdot \lambda$ is $P$-negative}),\\
0 & (\text{otherwise}).
\end{cases}
\]
\end{lem}
\begin{proof}
Assume that $n_w^{-1}\cdot \lambda$ is $P$-negative.
Then in particular it is $Q$-negative.
By Proposition~\ref{prop:I_P as A-module}, we have
\[
(\varphi E_{o_-}(\lambda))(T_{n_w})
=
\varphi(T_{n_w})e_Q(\sigma)(E^Q_{o_{-,Q}}(n_w^{-1}\cdot\lambda)).
\]
Since $n_w^{-1}\cdot \lambda$ is $P$-negative, we have $E^Q_{o_{-,Q}}(n_w^{-1}\cdot\lambda)\in \mathcal{H}_P^{Q-}$.
Hence $E^Q_{o_{-,Q}}(n_w^{-1}\cdot\lambda) = j_P^{Q-*}(E^P_{o_{-,P}}(n_w^{-1}\cdot\lambda))$ by \cite[Lemma~2.6]{arXiv:1612.01312}.
Therefore we have $e_Q(\sigma)(E^Q_{o_{-,Q}}(n_w^{-1}\cdot\lambda)) = \sigma(E^P_{o_{-,P}}(n_w^{-1}\cdot\lambda))$.
We get the lemma in this case.

Assume that $n_w^{-1}\cdot \lambda$ is not $P$-negative.
Then there exists $\alpha\in \Sigma^+\setminus\Sigma_P^+$ such that $\langle w(\alpha),\nu(\lambda)\rangle < 0$.
Take $\lambda_P^-$ as in Proposition~\ref{prop:localization as Levi subalgebra}  and put $\lambda_0 = n_w\cdot \lambda_P^-$.
We have $\langle w(\alpha),\nu(\lambda_0)\rangle = \langle \alpha,\nu(\lambda_P^-)\rangle > 0$.
Hence $\lambda$ and $\lambda_0$ are not in the same chamber.
Therefore we have $E_{o_-}(\lambda)E_{o_-}(\lambda_0) = 0$ by \eqref{eq:product formula} and \cite[Lemma~2.11]{arXiv:1612.01312}.
Hence we have $(\varphi E_{o_-}(\lambda)E_{o_-}(\lambda_0))(T_{n_w}) = 0$.
Since $n_w^{-1}\cdot \lambda_0 = \lambda_P^-$ is $P$-negative, we have $(\varphi E_{o_-}(\lambda)E_{o_-}(\lambda_0))(T_{n_w}) = (\varphi E_{o_-}(\lambda))(T_{n_w})\sigma(E^P_{o_{-,P}}(\lambda_P^-))$ as we have already proved.
Since $\lambda_P^-\in Z(W_P(1))$, $E^P_{o_{-,P}}(\lambda_P^-)$ is invertible.
Hence we have $(\varphi E_{o_-}(\lambda))(T_{n_w}) = 0$.
\end{proof}

\subsection{Module $\sigma_{\ell - \ell_P}$}
Let $P$ be a parabolic subgroup and $\sigma$ an $\mathcal{H}_P$-module.
Define a linear map $\sigma_{\ell - \ell_P}$ by $\sigma_{\ell - \ell_P}(T_w) = (-1)^{\ell(w) - \ell_P(w)}\sigma(T_w)$.
Then this defines a new $\mathcal{H}_P$-module $\sigma_{\ell - \ell_P}$~\cite[Lemma~4.1]{arXiv:1612.01312}.

\subsection{Supersingular modules}\label{subsec:supersingulars}
Assume that $q_s = 0$ for any $s\in S_\aff$.
Let $\mathcal{O}$ be a conjugacy class in $W(1)$ which is contained in $\Lambda(1)$.
For a spherical orientation $o$, set $z_\mathcal{O} = \sum_{\lambda\in \mathcal{O}}E_o(\lambda)$.
Then this does not depend on $o$ and gives an element of the center of $\mathcal{H}$ \cite[Theorem~5.1]{Vigneras-prop-III}.
The length of $\lambda\in \mathcal{O}$ does not depend on $\lambda$.
We denote it by $\ell(\mathcal{O})$.
\begin{defn}
Let $\pi$ be an $\mathcal{H}$-module.
We call $\pi$ supersingular if there exists $n\in \Z_{>0}$ such that $\pi z_\mathcal{O}^n = 0$ for any $\mathcal{O}$ such that $\ell(\mathcal{O}) > 0$.
\end{defn}
The simple supersingular $\mathcal{H}$-modules are classified in \cite{MR3263136,Vigneras-prop-III}.
We recall their results.
Assume that $C$ is a field.
Let $\chi$ be a character of $Z_\kappa\cap W_\aff(1)$ and put $S_{\aff,\chi} = \{s\in S_\aff\mid \chi(c_{\widetilde{s}}) \ne 0\}$ where $\widetilde{s}\in W(1)$ is a lift of $s\in S_\aff$.
Note that if $\widetilde{s}'$ is another lift, then $\widetilde{s}' = t\widetilde{s}$ for some $t\in Z_\kappa$.
Hence $\chi(c_{\widetilde{s}'}) = \chi(t)\chi(c_{\widetilde{s}})$.
Therefore the condition does not depend on a choice of a lift.
Let $J\subset S_{\aff,\chi}$.
Then the character $\Xi = \Xi_{J,\chi}$ of $\mathcal{H}_\aff$ is defined by
\begin{align*}
\Xi_{J,\chi}(T_t) & = \chi(t) \quad (t\in Z_\kappa\cap W_\aff(1)),\\
\Xi_{J,\chi}(T_{\widetilde{s}}) & = 
\begin{cases}
\chi(c_{\widetilde{s}}) & (s\in S_{\aff,\chi}\setminus J),\\
0 & (s\notin S_{\aff,\chi}\setminus J)
\end{cases}
\end{align*}
where $\widetilde{s}\in W_\aff(1)$ is a lift of $s$.
Let $\Omega(1)_{\Xi}$ be the stabilizer of $\Xi$ and $V$ an simple $C[\Omega(1)_{\Xi}]$-module such that $V|_{Z_\kappa\cap W_\aff(1)}$ is a direct sum of $\chi$.
Put $\mathcal{H}_\Xi = \mathcal{H}_\aff C[\Omega(1)_{\Xi}]$.
This is a subalgebra of $\mathcal{H}$.
For $X\in \mathcal{H}_\aff$ and $Y\in C[\Omega(1)_{\Xi}]$, we define the action of $XY$ on $\Xi\otimes V$ by $x\otimes y\mapsto xX\otimes yY$.
Then this defines a well-defined action of $\mathcal{H}_\Xi$ on $\Xi\otimes V$.
Set $\pi_{\chi,J,V} = (\Xi\otimes V)\otimes_{\mathcal{H}_\Xi}\mathcal{H}$.
\begin{prop}[{\cite[Theorem~1.6]{Vigneras-prop-III}}]
The module $\pi_{\chi,J,V}$ is simple and it is supersingular if and only if the groups generated by $J$ and generated by $S_{\aff,\chi}\setminus J$ are both finite.
If $C$ is an algebraically closed field, then any simple supersingulare modules are given in this way.
\end{prop}
\begin{rem}
This classification result is valid even though the data which defines $\mathcal{H}$ does not come from a reductive group.
\end{rem}

\subsection{Simple modules}
Assume that $C$ is an algebraically closed field of characteristic $p$.
We consider the following triple $(P,\sigma,Q)$.
\begin{itemize}
\item $P$ is a parabolic subgroup.
\item $\sigma$ is an simple supersingular $\mathcal{H}_P$-module.
\item $Q$ is a parabolic subgroup between $P$ and $P(\sigma)$.
\end{itemize}
Define
\[
I(P,\sigma,Q) = I_{P(\sigma)}(\St_Q^{P(\sigma)}(\sigma)).
\]
\begin{thm}[{\cite[Theorem~1.1]{arXiv:1406.1003_accepted}}]
The module $I(P,\sigma,Q)$ is simple and any simple module has this form.
Moreover, $(P,\sigma,Q)$ is unique up to isomorphism.
\end{thm}
Let $\chi$ be a character of $Z_\kappa\cap W_{\aff,P}(1)$, $J\subset S_{\aff,P,\chi}$ and $V$ an simple module of $C[\Omega_P(1)_{\Xi_{J,\chi}}]$ whose restriction to $Z_\kappa\cap W_{\aff,P}(1)$ is a direct sum of $\chi$.
Assume that the group generated by $J$ and generated by $S_{P,\aff,\chi}\setminus J$ are finite.
Then we put $I(P;\chi,J,V;Q) = I(P,\pi_{\chi,J,V},Q)$.
This is a simple module.

\subsection{Assumption on $C$}
In the rest of this paper, we always assume that $p = 0$ in $C$ unless otherwise stated since almost all results in this paper is proved only under this assumption.
Since $q_s$ is a power of $p$, this assumption implies $q_w = 0$ for any $w\in W(1)$ such that $\ell(w) > 0$.
When we discuss about simple modules, we also assume that $C$ is a field.
Such assumptions are written at the top of subsections or the statement of theorems.

\section{Twist}\label{sec:Twist}
We define an involution $\iota = \iota_G\colon \mathcal{H}\to \mathcal{H}$ by $\iota(T_w) = (-1)^{\ell(w)}T_w^*$~\cite[Proposition~4.23]{MR3484112} and $\pi^\iota = \pi\circ \iota$ for an $\mathcal{H}$-module $\pi$.
Obviously, $\pi^\iota$ is simple if $\pi$ is simple.
In this section, we calculate $\pi^\iota$ for simple modules $\pi$.

\subsection{Parabolic induction}
Let $P$ be a parabolic subgroup.
Then by \cite[Lemma~4.2]{arXiv:1612.01312}, we have $(\sigma^{\iota_P})_{\ell - \ell_P} = (\sigma_{\ell - \ell_P})^{\iota_P}$.
We denote this module by $\sigma_{\ell - \ell_P}^{\iota_P}$.
We have $I_P(\sigma)^\iota\simeq I'_P(\sigma^{\iota_P}_{\ell - \ell_P})$~\cite[Proposition~4.11]{arXiv:1612.01312}.
In this subsection, we prove the following proposition.
\begin{prop}\label{prop:hom between I and I'}
There exists a homomorphism $\Phi\colon I_P\to I'_P$ which is characterized by $\Phi(\varphi)(XT_{n_{w_Gw_P}}) = \varphi(XT_{n_{w_Gw_P}})$ for any $X\in \mathcal{H}$ and $\varphi\in I_P$.
\end{prop}
\begin{rem}
Assume that an $\mathcal{H}_P$-module $\sigma$ is not zero and $\varphi\in I_P(\sigma)$.
Then we have $\Phi(\varphi)(T_{n_{w_Gw_P}}) = \varphi(T_{n_{w_Gw_P}})$.
Since $I_P(\sigma)\to \sigma$ defined by $\varphi\mapsto \varphi(T_{n_{w_Gw_P}})$ is surjective by Proposition~\ref{prop:decomposition of I_P}, there exists $\varphi\in I_P(\sigma)$ such that $\varphi(T_{n_{w_Gw_P}})\ne 0$.
Hence $\Phi(\varphi)\ne 0$.
Therefore if $\sigma$ is not zero, then $\Phi\ne 0$.
\end{rem}

The following corollary is the first step to calculate $\pi^\iota$ for a simple $\mathcal{H}$-module $\pi$.

\begin{cor}\label{cor:I=I' for simples}
Assume that $C$ is a field.
Let $\sigma$ be an $\mathcal{H}_P$-module.
The representation $I_P(\sigma)$ is simple if and only if $I'_P(\sigma)$ is simple.
Moreover, if it is the case, then $I_P(\sigma)\simeq I'_P(\sigma)$.
Therefore if $I_P(\sigma)$ is simple, then $I_P(\sigma)^\iota \simeq I_P(\sigma^{\iota_P}_{\ell - \ell_P})$.
\end{cor}
\begin{proof}
If $I_P(\sigma)$ is simple, then the homomorphism in Proposition~\ref{prop:hom between I and I'} is injective.
We prove that $\dim I_P(\sigma) = \dim I'_P(\sigma) < \infty$.
Since $I_P(\sigma)$ is simple, $\sigma$ is also simple.
Hence it is finite-dimensional.
By $I_P(\sigma)\simeq \bigoplus_{w\in W^P_0}\sigma$ (Proposition~\ref{prop:decomposition of I_P}), we have $\dim I_P(\sigma) = \#W_0^P\dim \sigma$.
We also have $\dim I'_P(\sigma) = \#W_0^P\dim \sigma$ by \cite[Proposition~4.12]{arXiv:1612.01312}.
\end{proof}

For the proof of Proposition~\ref{prop:hom between I and I'}, by ~\cite[Proposition~4.13]{arXiv:1612.01312}, it is sufficient to prove the following lemma.
\begin{lem}
Put $P' = n_{w_Gw_P}\opposite{P}n_{w_Gw_P}^{-1}$.
The map $\varphi\mapsto (X\mapsto \varphi(XT_{n_{w_Gw_P}}))$ gives a homomorphism
\[
I_P(\sigma)\to \Hom_{(\mathcal{H}_{P'}^+,j_{P'}^+)}(\mathcal{H},n_{w_Gw_P}\sigma).
\]
\end{lem}
\begin{proof}
Set $n = n_{w_Gw_P}$ and we prove $I_P(\sigma)\ni\varphi\mapsto \varphi(T_n)\in n\sigma$ is $(\mathcal{H}_{P'}^+,j_{P'}^+)$-module homomorphism.
Let $w\in W_{P'}(1)$ be a $P'$-positive element.
By (4.1) in \cite{arXiv:1612.01312}, we have $j_{P'}^+(E_{o_{+,P'}}^{P'}(w))T_n  = T_nj_P^-(E^P_{o_{+,P}}(n^{-1}wn))$.
Hence
\begin{align*}
(\varphi j_{P'}^+(E_{o_{+,P'}}^{P'}(w)))(T_n) & = 
\varphi(j_{P'}^+(E_{o_{+,P'}}^{P'}(w))T_n)\\
& = \varphi(T_nj_P^-(E^P_{o_{+,P}}(n^{-1}wn))).
\end{align*}

We prove the following claim.
From this claim, $(\varphi j_{P'}^+(E_{o_{+,P'}}^{P'}(w)))(T_n)$ only depends on $\varphi(T_n)$ and $w$.
\begin{claim}
Let $\varphi\in I_P(\sigma)$ and $X\in \mathcal{H}$.
Then $\varphi(T_nX)$ only depends on $\varphi(T_n)$ and $X$.
\end{claim}
We introduce a basis defined by
\[
E_-(n_w \lambda) = q_{n_w \lambda}^{1/2}q_{n_w}^{-1/2}q_\lambda^{-1/2}T_{n_w}^* E_{o_-}(\lambda)
\]
for $w\in W_0$ and $\lambda\in\Lambda(1)$.
By \cite[Lemma~4.2]{arXiv:1406.1003_accepted}, $\{E_-(w)\mid w\in W(1)\}$ is a $C$-basis of $\mathcal{H}$.

To prove the claim, we may assume $X = E_-(n_w\lambda)$ for $w\in W_0$ and $\lambda\in \Lambda(1)$.
Take $\lambda_P^-\in \Lambda(1)$ as in Proposition~\ref{prop:localization as Levi subalgebra}  such that $\lambda \lambda_P^-$ is $P$-negative.
Then we have 
\[
\varphi(T_nE_-(n_w\lambda)) = \varphi(T_nE_-(n_w\lambda)E_{o_-}(\lambda_P^-))\sigma(E^P_{o_{-,P}}(\lambda_P^-))^{-1}
\]
by \cite[Lemma~2.6]{arXiv:1612.01312}.
If $\ell(n_w\lambda) + \ell(\lambda_P^-) > \ell(n_w\lambda \lambda_P^-)$, then $E_-(n_w\lambda)E_{o_-}(\lambda_P^-) = 0$.
Hence we have $\varphi(T_nE_-(n_w\lambda)) = 0$, so we get the claim.
If $\ell(n_w\lambda) + \ell(\lambda_P^-) = \ell(n_w\lambda \lambda_P^-)$, then $E_-(n_w\lambda)E_{o_-}(\lambda_P^-) = E_-(n_w\lambda \lambda_P^-)$.
Let $w_1\in W_0^P$ and $w_2\in W_{0,P}$ such that $w = w_1w_2$.
Then $n_{w_2}\lambda \lambda_P^-\in W_P(1)$ is $P$-negative.
Hence $\ell(n_{w_1}n_{w_2}\lambda \lambda_P^-) = \ell(n_{w_1}) + \ell(n_{w_2}\lambda \lambda_P^-)$ by \cite[Lemma~2.18]{arXiv:1612.01312}.
Therefore we have $E_-(n_w\lambda\lambda_P^-) = T_{n_{w_1}}^*E_-(n_ {w_2}\lambda\lambda_P^-)$.
If $w_1\ne  1$, then $w_1\notin W_{0,P}$.
Hence in a reduced expression of $w_1$, a simple reflection $s_\alpha$ for some $\alpha\in\Delta\setminus\Delta_P$ appears.
Therefore, there exists $x\in W_{0,P}$, $\alpha\in\Delta\setminus\Delta_P$ and $y\in W_0$ such that $w_1 = xs_\alpha y$ and $\ell(w_1) = \ell(x) + \ell(s_\alpha) + \ell(y)$.
Since $x\in W_{0,P}$, $x(\alpha)\in \Sigma^+\setminus\Sigma_P^+$.
Hence $w_Px(\alpha) > 0$.
Therefore $w_Gw_Px(\alpha) < 0$.
Hence $nn_xn_{s_\alpha} < nn_x$.
Therefore we have
\begin{align*}
\ell(n) + \ell(w_1) & = \ell(w_Gw_P) + \ell(x) + \ell(s_\alpha) + \ell(y)\\
& \ge \ell(w_Gw_Px) + \ell(s_\alpha) + \ell(y)\\
& > \ell(w_Gw_Pxs_\alpha) + \ell(y)\\
& \ge \ell(w_Gw_Pxs_\alpha y) = \ell(nn_{w_1}).
\end{align*}
Hence $T_nT_{n_{w_1}}^* = E_{o_+\cdot n^{-1}}(n)E_{o_+}(n_{w_1}) = 0$ by \cite[Example~5.22]{MR3484112} and \eqref{eq:product formula}.
Therefore, if $w_1\ne 1$, namely $w\notin W_{0,P}$, then we have $\varphi(T_nE_-(n_w\lambda\lambda_P^-)) = \varphi(T_nT_{n_{w_1}}^*E_-(n_{w_2}\lambda\lambda_P^-)) = 0$ again.
If $w\in W_{0,P}$, then $n_w\lambda \lambda_P^-$ is a $P$-negative element.
Hence $E^P_-(n_w\lambda\lambda_P^-)\in \mathcal{H}_P^-$ and we have $E_-(n_w\lambda\lambda_P^-) = j_P^{-*}(E_-^P(n_w\lambda \lambda_P^-))$ by \cite[Lemma~4.6]{arXiv:1406.1003_accepted}.
Therefore we have 
\begin{align*}
\varphi(T_nE_-(n_w\lambda)) & = \varphi(T_nE_-(n_w\lambda\lambda_P^-))\sigma(E^P_{o_{-,P}}(\lambda_P^-))^{-1}\\
& = \varphi(T_n)\sigma(E_-^P(n_w\lambda\lambda_P^-)E^P_{o_{-,P}}(\lambda_P^-)^{-1})\\
& = \varphi(T_n)\sigma(E_-^P(n_w\lambda)).
\end{align*}
The claim is proved.

Let $\varphi_0\in I_P(\sigma)$ be such that $\varphi_0(T_n) = \varphi(T_n)$ and $\varphi_0(T_{n_v}) = 0$ for $v\in W_0^P\setminus\{w_Gw_P\}$.
Then, as a consequence of the claim, we have
\[
(\varphi j_{P'}^+(E_{o_{+,P'}}^{P'}(w)))(T_n)
=
(\varphi_0 j_{P'}^+(E_{o_{+,P'}}^{P'}(w)))(T_n).
\]
By the proof of \cite[Proposition~4.14]{arXiv:1406.1003_accepted}, we have
\[
(\varphi_0 j_{P'}^+(E_{o_{+,P'}}^{P'}(w)))(T_n)
=
\varphi_0(T_n)(n\sigma)(E_{o_+,P'}^{P'}(w)).
\]
Since $\varphi_0(T_n) = \varphi(T_n)$, we have
\[
\varphi_0(T_n)(n\sigma)(E_{o_+,P'}^{P'}(w))
=
\varphi(T_n)(n\sigma)(E_{o_+,P'}^{P'}(w)).
\]
Hence
\[
(\varphi j_{P'}^+(E_{o_{+,P'}}^{P'}(w)))(T_n)
=
\varphi(T_n)(n\sigma)(E_{o_+,P'}^{P'}(w)).
\]
We get the lemma.
\end{proof}
Here is the compatibility with the transitivity of $I_P$ and $I'_P$~\cite[Proposition~4.12]{arXiv:1612.01312}.
\begin{lem}\label{lem:transitivity of I->I'}
Let $Q\supset P$ be a parabolic subgroups. 
Then the following three maps are equal.
\begin{enumerate}
\item $I_P\to I'_P$.
\item $I_P = I_Q\circ I_P^Q\to I'_Q\circ I_P^Q\to I'_Q\circ I^{Q\prime}_P = I'_P$.
\item $I_P = I_Q\circ I_P^Q\to I_Q\circ I_P^{Q\prime}\to I'_Q\circ I^{Q\prime}_P = I'_P$.
\end{enumerate}
\end{lem}
\begin{proof}
Let $\sigma$ be an $\mathcal{H}_P$-module.
In each cases, let $\varphi\in I_P(\sigma)$ and $\psi_i\in I'_P(\sigma)$ the image of $\varphi$ by the map in $(i)$ for $i = 1,2,3$.
Then $\psi_1$ is characterized by $\varphi(XT_{n_{w_Gw_P}}) = \psi_1(XT_{n_{w_Gw_P}})$ for any $X\in \mathcal{H}$.
We denote the corresponding element to $\varphi\in I_P(\sigma)$ (resp.\ $\psi_i\in I'_P(\sigma)$) by $\varphi'\in (I_Q\circ I_P^Q)(\sigma)$ (resp.\ $\psi_i'\in (I'_Q\circ I^{Q\prime}_P)(\sigma)$).

We consider $\psi_2$.
We have $\psi_2(XT_{n_{w_Gw_P}}) = \psi'_2(XT_{n_{w_Gw_P}})(1)$.
Since $\ell(w_Gw_P) = \ell(w_Gw_Q) + \ell(w_Qw_P)$, we have $T_{n_Gw_P} = T_{n_{w_Gw_Q}}T_{n_{w_Qw_P}}$.
Hence $\psi'_2(XT_{n_{w_Gw_P}})(1) = \psi'_2(XT_{n_{w_Gw_Q}}T_{n_{w_Qw_P}})(1)$.
Since $T^Q_{n_{w_Qw_P}}\in \mathcal{H}_Q^-$ and $j_Q^-(T^Q_{n_{w_Qw_P}}) = T_{n_{w_Qw_P}}$~ \cite[Lemma~2.6]{arXiv:1612.01312}, we have $\psi'_2(XT_{n_{w_Gw_Q}}T_{n_{w_Qw_P}})(1) = \psi'_2(XT_{n_{w_Gw_Q}})(T^Q_{n_{w_Qw_P}})$.
Let $\psi''_2\in (I'_Q\circ I_P^Q)(\sigma)$ be the image of $\varphi$.
Then $\psi''_2(XT_{n_{w_Gw_Q}})(Y) = \varphi'(XT_{n_{w_Gw_Q}})(Y)$ and $\psi''_2(X)(YT^Q_{n_{w_Qw_P}}) = \psi'_2(X)(YT^Q_{n_{w_Qw_P}})$ for $X\in \mathcal{H}$ and $Y\in \mathcal{H}_Q$.
Therefore we have $\psi_2(XT_{n_{w_Gw_P}}) = \psi''_2(XT_{n_{w_Gw_P}})(T^Q_{n_{w_Qw_P}}) = \varphi'(XT_{n_{w_Gw_P}})(T^Q_{n_{w_Qw_P}})$.
Again, since $T^Q_{n_{w_Qw_P}}\in \mathcal{H}_Q^-$ and $j_{Q}^{-*}(T_{n_{w_Qw_P}}^Q) = T_{n_{w_Qw_P}}$ \cite[Lemma~2.6]{arXiv:1612.01312}, we have $\varphi'(XT_{n_{w_Gw_Q}})(T^Q_{n_{w_Qw_P}}) = \varphi'(XT_{n_{w_Gw_Q}}T_{n_{w_Qw_P}})(1) = \varphi'(XT_{n_{w_Gw_P}})(1) = \varphi(XT_{n_{w_Gw_P}})$.
Hence the map in (2) satisfies the characterization of the map in (1).

The proof for (3) is similar.
Let $\psi''_3\in (I_Q\circ I_P^{Q\prime})(\sigma)$ be the image of $\varphi$.
Then we have
\begin{align*}
\psi_3(XT_{n_{w_Gw_P}})  & = \psi'_3(XT_{n_{w_Gw_P}})(1)\\
& = \psi'_3(XT_{n_{w_Gw_Q}}T_{n_{w_Qw_P}})(1)\\
& = \psi'_3(XT_{n_{w_Gw_Q}})(T^Q_{n_{w_Qw_P}})\\
& = \psi''_3(XT_{n_{w_Gw_Q}})(T^Q_{n_{w_Qw_P}})\\
& = \varphi'(XT_{n_{w_Gw_Q}})(T^Q_{n_{w_Qw_P}})\\
& = \varphi'(XT_{n_{w_Gw_Q}}T_{n_{w_Qw_P}})(1)\\
& = \varphi'(XT_{n_{w_Gw_P}})(1)\\
& = \varphi(XT_{n_{w_Gw_P}}).
\end{align*}
This finish the proof.
\end{proof}

\subsection{Steinberg modules}\label{subsec:twist:Steinberg modules}
Next we consider the twist of the generalized Steinberg modules.
Until subsection \ref{subsec:exactness}, we keep the following settings.
\begin{itemize}
\item $P$ is a parabolic subgroup such that $\Delta_P$ is orthogonal to $\Delta\setminus\Delta_P$.
\item Let $P_2$ be a parabolic subgroup corresponding to $\Delta\setminus\Delta_P$.
\item $\sigma$ is an $\mathcal{H}_P$-module such that $P(\sigma) = G$.
\end{itemize}

We prove the following proposition.
\begin{thm}\label{thm:twist of Steinberg}
Let $Q$ be a parabolic subgroup containing $P$ and $Q^c$ the parabolic subgroup corresponding to $\Delta_P\cup (\Delta\setminus\Delta_Q)$.
Then we have $\St_Q(\sigma)^\iota \simeq \St_{Q^c}(\sigma^{\iota_P}_{\ell - \ell_P})$.
\end{thm}
The proof of this proposition continues until subsection~\ref{subsec:exactness}.
In this subsection, we prove this proposition for $Q = P$.
\begin{lem}\label{lem:twist of extension}
Theorem~\ref{thm:twist of Steinberg} is true if $Q = P$.
\end{lem}
\begin{proof}
We use \cite[Proposition~3.12]{arXiv:1612.01312}.
Let $w\in W_P(1)$ and assume that $w$ is $P$-positive.
Then $(\St_P(\sigma))^\iota(T_w^*) = (-1)^{\ell(w)}\St_P(\sigma)(T_w) = (-1)^{\ell(w)}\sigma(T^P_w) = (-1)^{\ell(w) - \ell_P(w)}\sigma^{\iota_P}(T^{P*}_w) = \sigma_{\ell - \ell_P}^{\iota_p}(T_w^{P*})$.
On the other hand, for $w\in W_{\aff,P_2}(1)$, we have $(\St_P(\sigma))^\iota(T_w^*) = (-1)^{\ell(w)}\St_P(\sigma)(T_w) = 1$.
Therefore, by a characterization of the extension, we have $(\St_P(\sigma))^\iota\simeq e_G(\sigma^{\iota_P}_{\ell - \ell_P}) = \St_G(\sigma_{\ell - \ell_P}^{\iota_P})$.
Since $P^c = G$, we get the lemma.
\end{proof}

\subsection{An exact sequence}
We express $\St_{Q^c}(\sigma)$ as the kernel of a certain homomorphism.
As a consequence, we deduce Theorem~\ref{thm:twist of Steinberg} to the exactness of a certain sequence (Lemma~\ref{lem:exact sequence related to I and I'}).

Let $Q$ be a parabolic subgroup containing $P$.
Then we have an exact sequence
\[
0\to \sum_{Q\supset R\supsetneq P}I^Q_R(e_R(\sigma))\to I^Q_P(\sigma)\to \St_P^Q(\sigma)\to 0.
\]
Applying $I_Q$ and using the transitivity~\cite[Proposition~4.10]{MR3437789}, we have an exact sequence
\[
0\to \sum_{Q\supset R\supsetneq P}I_R(e_R(\sigma))\to I_P(\sigma)\to I_Q(\St_P^Q(\sigma))\to 0.
\]
Let $Q_1$ be a parabolic subgroup containing $Q$.
Then we have $\sum_{Q\supset R\supsetneq P}I_R(e_R(\sigma))\subset \sum_{Q_1\supset R\supsetneq P}I_R(e_R(\sigma))$.
Hence we have a homomorphism $I_Q(\St_P^Q(\sigma))\to I_{Q_1}(\St_P^{Q_1}(\sigma))$ which makes the following diagram commutative.
\[
\begin{tikzcd}
0\arrow[r] & \sum_{Q\supset R\supsetneq P}I_R(e_R(\sigma))\arrow[r]\arrow[d,hookrightarrow] & I_P(\sigma)\arrow[r]\arrow[d,equal] & I_Q(\St_P^Q(\sigma))\arrow[r]\arrow[d,dashed] & 0\\
0\arrow[r] & \sum_{Q_1\supset R\supsetneq P}I_R(e_R(\sigma))\arrow[r] & I_P(\sigma)\arrow[r] & I_{Q_1}(\St_P^{Q_1}(\sigma))\arrow[r] & 0.
\end{tikzcd}
\]
\begin{lem}\label{lem:St as kernel}
With the above homomorphisms, we have an exact sequence
\[
0\to \St_{Q^c}(\sigma)\to I_Q(\St_P^Q(\sigma))\to \bigoplus_{Q_1\supsetneq P}I_{Q_1}(\St_P^{Q_1}(\sigma)).
\]
\end{lem}
We will prove this lemma at the end of this subsection.
Applying $\iota$ to the exact sequence in the lemma, we have
\[
0\to \St_{Q^c}(\sigma)^\iota\to I_Q(\St_P^Q(\sigma))^\iota\to \bigoplus_{Q_1\supsetneq P}I_{Q_1}(\St_P^{Q_1}(\sigma))^\iota.
\]
Using Lemma~\ref{lem:twist of extension} and \cite[Lemma~4.9, Proposition~4.11]{arXiv:1612.01312}, we have $I_Q(\St_P^Q(\sigma))^\iota = I'_Q((\St_P^Q(\sigma))^{\iota_Q}_{\ell - \ell_Q}) = I'_Q(e_Q(\sigma^{\iota_P}_{\ell_Q - \ell_P})_{\ell - \ell_Q}) = I'_Q(e_Q(\sigma^{\iota_P}_{\ell - \ell_P}))$.
We get the following exact sequence
\[
0\to (\St_{Q^c}\sigma)^\iota\to I'_Q(e_Q(\sigma_{\ell - \ell_P}^{\iota_P}))\to \bigoplus_{Q_1\supsetneq Q}I'_{Q_1}(e_{Q_1}(\sigma_{\ell - \ell_P}^{\iota_P})).
\]
Therefore Theorem~\ref{thm:twist of Steinberg} follows from the following lemma.
\begin{lem}\label{lem:exact sequence related to I and I'}
The following sequence is exact.
\[
\bigoplus_{Q_1\supsetneq Q}I_{Q_1}(e_{Q_1}(\sigma))\to I_Q(e_Q(\sigma))\to I'_Q(e_Q(\sigma)) \to \bigoplus_{Q_1\supsetneq Q}I'_{Q_1}(e_{Q_1}(\sigma)).
\]
Here $I_Q(e_Q(\sigma))\to I'_Q(e_Q(\sigma))$ is given in Proposition~\ref{prop:hom between I and I'} and $I_{Q_1}(e_{Q_1}(\sigma))\to I_Q(e_Q(\sigma))$ is the natural embedding.
\end{lem}

As the end of this subsection, we prove Lemma~\ref{lem:St as kernel}.
We need the following lemma.
For parabolic subgroups $P_1,\dots,P_r$, let $\langle P_1,\dots,P_r\rangle$ be the parabolic subgroup generated by $P_1,\dots,P_r$.
In other words, $\langle P_1,\dots,P_r\rangle$ is the parabolic subgroup corresponding to $\Delta_{P_1}\cup \dots\cup \Delta_{P_r}$.
\begin{lem}\label{lem:intersection and union of parabolic inductions}
Let $\mathcal{P}_1,\mathcal{P}_2\subset \{Q\mid Q\supset P\}$.
Then, in $I_P(\sigma)$, we have 
\[
\left(\sum_{Q_1\in \mathcal{P}_1}I_{Q_1}(e_{Q_1}(\sigma))\right)\cap \left(\sum_{Q_2\in \mathcal{P}_2}I_{Q_2}(e_{Q_2}(\sigma))\right)
=
\sum_{Q_1\in \mathcal{P}_1,Q_2\in \mathcal{P}_2}I_{\langle Q_1,Q_2\rangle}(e_{\langle Q_1,Q_2\rangle}(\sigma)).
\]
\end{lem}

\begin{rem}
We prove Lemma~\ref{lem:St as kernel} and \ref{lem:intersection and union of parabolic inductions} for any commutative ring $C$.
\end{rem}

\begin{proof}
Note that $\Delta_P$ is orthogonal to $\Delta_{P_2}$ as we assumed.
Therefore we have $W_0^P = W_{0,P_2}$ and for any parabolic subgroup $Q$ containing $P$, we have $W_{0,P_2} = W_0^{Q}W_{0,Q\cap P_2}$.

Put $I_Q = I_Q(e_Q(\sigma))$.
We prove the lemma by induction on $\#\mathcal{P}_1$.
Assume that $\#\mathcal{P}_1 = 1$.
By \cite[Lemma~3.8]{arXiv:1612.01312}, it is sufficient to prove that $I_{Q_1}\cap I_{Q_2} = I_{\langle Q_1,Q_2\rangle}$.
Obviously we have $I_{Q_1}\cap I_{Q_2} \supset I_{\langle Q_1,Q_2\rangle}$.
Let $\varphi\in I_{Q_1}\cap I_{Q_2}$.
Then for $w\in W^{\langle Q_1,Q_2\rangle}_0$ and $\alpha\in \Delta_{Q_1}\setminus \Delta_P$, we have $\varphi(T_{n_{ws_\alpha}}) = \varphi(T_{n_{w}})e_{Q_1}(\sigma)(T^{Q_1}_{n_{s_\alpha}})$.
Since $n_{s_\alpha}\in W_{P_2\cap Q,\aff}$, we have $e_{Q_1}(\sigma)(T^{Q_1}_{n_{s_\alpha}}) = q_{s_\alpha}$.
Therefore we have $\varphi(T_{n_{ws_\alpha}}) = q_{s_\alpha}\varphi(T_{n_w})$.
This also holds for $\alpha\in\Delta_{Q_2}\setminus\Delta_P$.
Therefore $\varphi(T_{n_{wv}}) = q_v\varphi(T_{n_w})$ for any $v$ generated by $\{s_\alpha\mid \alpha\in (\Delta_{Q_1}\cup \Delta_{Q_2})\setminus\Delta_P\}$.

Since $(\Delta_{Q_1}\cup \Delta_{Q_2})\setminus\Delta_P = \Delta_{\langle Q_1,Q_2\rangle\cap P_2}$, the group generated by $\{s_\alpha\mid \alpha\in (\Delta_{Q_1}\cup \Delta_{Q_2})\setminus\Delta_P\}$ is $W_{0,\langle Q_1,Q_2\rangle\cap P_2}$.
Hence $\varphi(T_{n_{wv}}) = q_v\varphi(T_{n_w})$ for any $w\in W^{\langle Q_1,Q_2\rangle}_0$ and $v\in W_{0,\langle Q_1,Q_2\rangle\cap P_2}$.
Define $\varphi' \in I_{\langle Q_1,Q_2\rangle}(e_{\langle Q_1,Q_2\rangle}(\sigma))$ by $\varphi'(T_{n_w}) = \varphi(T_{n_w})$ for any $w\in W_0^{\langle Q_1,Q_2\rangle}$. (Such element uniquely exists by Proposition~\ref{prop:decomposition of I_P}.)
Then $\varphi'$ also satisfies $\varphi(T_{n_{wv}}) = q_v\varphi(T_{n_w})$ for any $w\in W^{\langle Q_1,Q_2\rangle}_0$ and $v\in W_{0,\langle Q_1,Q_2\rangle\cap P_2}$.
Hence $\varphi(T_{n_w}) = \varphi'(T_{n_w})$ for any $w\in W_0^{\langle Q_1,Q_2\rangle}W_{0,\langle Q_1,Q_2\rangle\cap P_2} = W_{0,P_2} = W_0^P$.
By Proposition~\ref{prop:decomposition of I_P}, $\varphi = \varphi' \in I_{\langle Q_1,Q_2\rangle}$.

Now we prove the general case.
Obviously we have
\[
\left(\sum_{Q_1\in \mathcal{P}_1}I_{Q_1}\right)\cap \left(\sum_{Q_2\in \mathcal{P}_2}I_{Q_2}\right)
\supset
\sum_{Q_1\in \mathcal{P}_1,Q_2\in \mathcal{P}_2}I_{\langle Q_1,Q_2\rangle}.
\]
We prove the reverse inclusion.
Take $f$ from the left hand side.
Fix $Q_0\in \mathcal{P}_1$ and put $\mathcal{P}'_1 = \mathcal{P}_1\setminus\{Q_0\}$.
Take $f_1\in I_{Q_0}$ and $f_2\in \sum_{Q_1\in \mathcal{P}'_1}I_{Q_1}$ such that $f = f_1 + f_2$.
Then we have
\[
f_2\in \left(\sum_{Q_2\in \mathcal{P}_2}I_{Q_2} + I_{Q_0}\right)\cap \sum_{Q_1\in \mathcal{P}_1'}I_{Q_1}.
\]
By inductive hypothesis, the right hand side is
\[
\sum_{Q_1\in \mathcal{P}'_1,Q_2\in \mathcal{P}_2}I_{\langle Q_1,Q_2\rangle} + \sum_{Q_1\in \mathcal{P}_1'}I_{\langle Q_0,Q_1\rangle}.
\]
Since $I_{\langle Q_0,Q_1\rangle}\subset I_{Q_0}$, we get
\[
f_2\in \sum_{Q_1\in \mathcal{P}'_1,Q_2\in \mathcal{P}_2}I_{\langle Q_1,Q_2\rangle} + I_{Q_0}.
\]
We have $f_1\in I_{Q_0}$.
Therefore
\[
f = f_1 + f_2\in \sum_{Q_1\in \mathcal{P}'_1,Q_2\in \mathcal{P}_2}I_{\langle Q_1,Q_2\rangle} + I_{Q_0}.
\]
Take $f'_1\in \sum_{Q_1\in \mathcal{P}'_1,Q_2\in \mathcal{P}_2}I_{\langle Q_1,Q_2\rangle}$ and $f'_2\in I_{Q_0}$ such that $f = f'_1 + f'_2$.
Then $f'_1\in \sum_{Q_2\in \mathcal{P}_2}I_{Q_2}$.
By the assumption, $f \in \sum_{Q_2\in \mathcal{P}_2}I_{Q_2}$.
Therefore we have $f'_2 = f - f'_1\in \sum_{Q_2\in \mathcal{P}_2}I_{Q_2}$.
Hence 
\[
f'_2\in I_{Q_0}\cap \sum_{Q_2\in \mathcal{P}_2}I_{Q_2} = \sum_{Q_2\in \mathcal{P}_2}I_{\langle Q_0,Q_2\rangle}.
\]
Here we use the lemma for $\mathcal{P}_1 = \{Q_0\}$.
Hence
\begin{align*}
f = f'_1 + f'_2 & \in \sum_{Q_1\in \mathcal{P}'_1,Q_2\in \mathcal{P}_2}I_{\langle Q_1,Q_2\rangle} + \sum_{Q_2\in \mathcal{P}_2}I_{\langle Q_0,Q_2\rangle}\\
& = \sum_{Q_1\in \mathcal{P}_1,Q_2\in \mathcal{P}_2}I_{\langle Q_1,Q_2\rangle}.
\end{align*}
We get the lemma.
\end{proof}
\begin{proof}[Proof of Lemma~\ref{lem:St as kernel}]
Let $Q_1\supsetneq Q$.
By the exact sequence before Lemma~\ref{lem:St as kernel}, the kernel of $I_P(\sigma)\to I_{Q}(\St_P^Q(\sigma))\to I_{Q_1}(\St_P^{Q_1}(\sigma))$ is $\sum_{Q_1\supset R\supsetneq P}I_R(e_R(\sigma))$.
Since $\Ker(I_P(\sigma)\to I_Q(\St_P^Q(\sigma)) = \sum_{Q\supset R\supsetneq P}I_R(e_R(\sigma))$, we have
\begin{align*}
&\Ker\left(I_Q(\St^Q_P(\sigma))\to \bigoplus_{Q_1\supsetneq Q}I_{Q_1}(\St^{Q_1}_P(\sigma))\right)\\
& \simeq
\left.\Ker\left(I_P(\sigma)\to \bigoplus_{Q_1\supsetneq Q}I_{Q_1}(\St_P^{Q_1}(\sigma))\right)\;\middle/\;\Ker(I_P(\sigma)\to I_Q(\St_P^Q(\sigma)))\right.\\
& =
\left.\left(\bigcap_{Q_1\supsetneq Q}\Ker(I_P(\sigma)\to I_{Q_1}(\St_P^{Q_1}(\sigma)))\right)\;\middle/\;\Ker(I_P(\sigma)\to I_Q(\St_P^Q(\sigma)))\right.\\
& =
\left.\bigcap_{Q_1\supsetneq Q}\sum_{Q_1\supset R\supsetneq P}I_R(e_R(\sigma))\;\middle/\;\sum_{Q\supset R\supsetneq P}I_R(e_R(\sigma))\right..
\end{align*}
Set $A = \bigcap_{Q_1\supsetneq Q} \sum_{Q_1\supset R\supsetneq P}I_R(e_R(\sigma))$ and $B = \sum_{Q\supset R\supsetneq P}I_R(e_R(\sigma))$.
We prove the following.
\begin{enumerate}
\item $I_{Q^c}(e_{Q^c}(\sigma)) + B = A$.
\item $I_{Q^c}(e_{Q^c}(\sigma))\cap B = \sum_{R\supsetneq Q^c}I_R(e_R(\sigma))$.
\end{enumerate}
First we prove (1).
We prove $I_{Q^c}(e_{Q^c}(\sigma)) + B \subset A$.
Let $Q_1\supsetneq  Q$.
Take $\alpha\in\Delta_{Q_1}\setminus\Delta_Q$ and let $R$ be a parabolic subgroup corresponding to $\Delta_P\cup\{\alpha\}$.
Since $\alpha\in\Delta_{Q^c}$ and $\Delta_P\subset \Delta_{Q^c}$, we have $\Delta_R \subset \Delta_{Q^c}$.
Hence $I_{Q^c}(e_{Q^c}(\sigma))\subset I_R(e_R(\sigma))\subset A$.
Obviously we have $B\subset A$.

We prove the reverse inclusion.
By Lemma~\ref{lem:intersection and union of parabolic inductions}, we have
\[
A = \sum_{(R_{Q_1})_{Q_1\supsetneq Q}}I_{\langle R_{Q_1}\rangle_{Q_1}}(e_{\langle R_{Q_1}\rangle_{Q_1}}(\sigma)),
\]
where $R_{Q_1}$ satisfies $Q_1\supset R_{Q_1}\supsetneq P$ and $\langle R_{Q_1}\rangle_{Q_1}$ is the group generated by $\{R_{Q_1}\mid Q_1\supsetneq Q\}$.
Hence it is sufficient to prove that each $I_{\langle R_{Q_1}\rangle_{Q_1}}(e_{\langle R_{Q_1}\rangle_{Q_1}}(\sigma))$ is contained in $I_{Q^c}(e_{Q^c}(\sigma)) + B$.
If $Q\supset R_{Q_0}$ for some $Q_0\supsetneq Q$, then for such $Q_0$, we have  $I_{\langle R_{Q_1}\rangle_{Q_1}}(e_{\langle R_{Q_1}\rangle_{Q_1}}(\sigma))\subset I_{R_{Q_0}}(e_{R_{Q_0}}(\sigma))\subset B$.
Assume that $Q\not\supset R_{Q_1}$ for any $Q_1\supsetneq Q$ and we prove $\Delta_{Q^c}\subset \bigcup_{Q_1}\Delta_{R_{Q_1}}$.
Let $\alpha\in\Delta_{Q^c} = (\Delta\setminus\Delta_Q)\cup\Delta_P$.
If $\alpha\in\Delta_P$, then we have $\alpha\in\bigcup_{Q_1}\Delta_{R_{Q_1}}$ since $P\subset R_{Q_1}$.
Assume that $\alpha\in\Delta\setminus\Delta_Q$.
Let $Q_\alpha$ be the parabolic subgroup corresponding to $\Delta_Q\cup\{\alpha\}$.
Then by the assumption, $\Delta_{R_{Q_\alpha}}$ is contained in $\Delta_{Q_\alpha} = \Delta_Q\cup\{\alpha\}$ and is not contained in $\Delta_Q$.
Hence $\alpha\in\Delta_{R_{Q_\alpha}}$.
Therefore $\alpha\in \Delta_{R_{Q_\alpha}}\subset \bigcup_{Q_1\supsetneq Q}\Delta_{R_{Q_1}} = \Delta_{\langle R_{Q_1}\rangle_{Q_1}}$ by taking $Q_1 = Q_\alpha$.
Hence we have $Q^c\subset \langle R_{Q_1}\rangle_{Q_1}$.
Therefore $I_{\langle R_{Q_1}\rangle_{Q_1}}(e_{\langle R_{Q_1}\rangle_{Q_1}}(\sigma))\subset I_{Q^c}(e_{Q^c}(\sigma))$.
We get (1).

We prove (2).
By the above lemma again, we have
\[
I_{Q^c}(e_{Q^c}(\sigma))\cap B = \sum_{Q\supset R\supsetneq P}I_{\langle R,Q^c\rangle}(e_{\langle R,Q^c\rangle}(\sigma)).
\]
First we prove $I_{Q^c}(e_{Q^c}(\sigma))\cap B \subset \sum_{R_1\supsetneq Q^c}I_{R_1}(e_{R_1}(\sigma))$, namely,  for each $R$ such that $Q\supset R\supsetneq P$ we have $I_{\langle R,Q^c\rangle }(e_{\langle R,Q^c\rangle}(\sigma))\subset \sum_{R_1\supsetneq Q^c}I_{R_1}R(e_{R_1}(\sigma))$
For such $R$, we can take $\alpha\in\Delta_R\setminus\Delta_P$.
Since $\alpha\in\Delta_Q\setminus\Delta_P$, $\alpha\notin \Delta_{Q^c}$.
Hence $\Delta_{\langle R,Q^c\rangle} = \Delta_R\cup\Delta_{Q^c} \supset \{\alpha\}\cup\Delta_{Q^c}\supsetneq \Delta_{Q^c}$.
Therefore $\langle R,Q^c\rangle\supsetneq Q^c$.
Hence $I_{\langle R,Q^c\rangle}(e_{\langle R,Q^c\rangle}(\sigma))\subset \sum_{R_1\supsetneq Q^c}I_{R_1}(e_{R_1}(\sigma))$ by taking $R_1 = \langle R,Q^c\rangle$.

We prove $I_R(e_R(\sigma))\subset I_{Q^c}(e_{Q^c}(\sigma))\cap B$ for any $R$ such that $R\supsetneq Q^c$.
We can take $\alpha\in\Delta_R\setminus\Delta_{Q^c}$.
Let $P_\alpha$ be the parabolic subgroup corresponding to $\Delta_P\cup\{\alpha\}$.
Since $\alpha\notin\Delta_{Q^c}$, we have $\alpha\in\Delta_Q$.
Therefore $Q\supset P_\alpha \supsetneq P$.
Hence $\Delta_R\supset \Delta_{Q^c}\cup \{\alpha\} = \Delta_{\langle P_\alpha,Q^c\rangle}$.
Therefore $R\supset \langle P_\alpha,Q^c\rangle$.
Hence $I_R(e_R(\sigma))\subset I_{\langle P_\alpha,Q^c\rangle}(e_{\langle P_\alpha,Q^c\rangle}(\sigma))\subset I_{Q^c}(e_{Q^c}(\sigma))\cap B$.
We get (2) and the proof of the lemma is finished.
\end{proof}

\subsection{The kernel of $I_Q(e_Q(\sigma))\to I'_Q(e_Q(\sigma))$}
We determine $\Ker(I_Q(e_Q(\sigma))\to I'_Q(e_Q(\sigma)))$, namely we prove the following lemma.
\begin{lem}
Set $A = \{w\in W_0^Q\mid \Delta_w = \Delta_Q\}$.
Then we have
\begin{align*}
& \Ker(I_Q(e_Q(\sigma))\to I'_Q(e_Q(\sigma)))\\
& = \{\varphi\in I_Q(e_Q(\sigma))\mid \text{$\varphi(XT_{n_w}) = 0$ for any $X\in \mathcal{H}$ and $w\in A$}\}\\
& = \{\varphi\in I_Q(e_Q(\sigma))\mid \text{$\varphi(T_{n_w}) = 0$ for any $w\in A$}\} = \sum_{Q_1\supsetneq Q}I_{Q_1}(e_{Q_1}(\sigma)).
\end{align*}
\end{lem}
The last equality follows from Lemma~\ref{lem:embedding of smaller parabolic induction} and the following lemma.
\begin{lem}
We have $\bigcap_{Q_1\supsetneq Q}(W_0^Q\setminus W_0^{Q_1}) = \{w\in W_0^Q\mid \Delta_w = \Delta_Q\}$.
\end{lem}
\begin{proof}
Let $w\in W_0^Q$.
Then $w(\Delta_Q) > 0$ and we have $\Delta_w = \Delta_Q$ if and only if for any $Q_1\supsetneq Q$, $w(\Delta_{Q_1})\not\subset \Sigma^+$.
Since we have $w(\Delta_{Q_1})\not\subset \Sigma^+$ if and only if $w\notin W_0^{Q_1}$, we get the lemma.
\end{proof}

Let $\varphi\in I_Q(e_Q(\sigma))$ such that $\varphi(T_{n_w}) = 0$ for any $w\in A$ and take $X\in \mathcal{H}$.
The last equality implies that the set of such $\varphi$ is stable under $\mathcal{H}$.
Hence $\varphi X$ also satisfies the same condition.
Therefore $\varphi(XT_{n_w}) = 0$ for any $w\in A$.
Namely we get the second equality.

Let $\psi$ be the image of $\varphi$ under $I_Q(e_Q(\sigma))\to I'_Q(e_Q(\sigma))$.
Then $\psi$ is characterized by $\varphi(XT_{n_{w_Gw_Q}}) = \psi(XT_{n_{w_Gw_Q}})$.
Therefore we have
\begin{align*}
&\Ker(I_Q(e_Q(\sigma))\to I'_Q(e_Q(\sigma)))\\
&=
\{\varphi\in I_Q(e_Q(\sigma))\mid \text{$\varphi(XT_{n_{w_Gw_Q}}) = 0$ for any $X\in \mathcal{H}$}\}.
\end{align*}
Therefore, to prove the lemma, it is sufficient to prove the following lemma.
\begin{lem}
Let $\varphi\in I_Q(e_Q(\sigma))$.
Assume that $\varphi(XT_{n_{w_Gw_Q}}) = 0$ for any $X\in \mathcal{H}$.
Then we have $\varphi(XT_{n_w}) = 0$ for any $X\in \mathcal{H}$ and $w\in W_0^Q$ such that $\Delta_w = \Delta_Q$.
\end{lem}
\begin{proof}
We prove the lemma by backward induction on $\ell(w)$.
If $w \ne w_Gw_Q$, then there exists $\alpha\in\Delta$ such that $s_\alpha w > w$, $\Delta_w = \Delta_{s_\alpha w}$ and $w^{-1}(\alpha)$ is not simple~\cite[Lemma~3.15]{arXiv:1406.1003_accepted}.
Set $s = s_\alpha$.
Since $\Delta_{sw} = \Delta_w = \Delta_Q$, we have $sw\in W^Q_0$.
If $w^{-1}(\alpha)\in\Sigma^+_Q$, then since $sw\in W^Q_0$, we have $-\alpha = sw(w^{-1}(\alpha))\in \Sigma^+$.
This is a contradiction.
Hence 
\[w^{-1}(\alpha)\in\Sigma^+\setminus\Sigma^+_Q.\]

Take $\lambda_P^-\in Z(W_P(1))$ as in Proposition~\ref{prop:localization as Levi subalgebra} .
Put $\lambda = n_w\cdot (\lambda_P^-)^2$.
We  prove:
\begin{claim}
$E_-(\lambda n_s^{-1})(T_{n_s} - c_{n_s}) = E_{o_-}(\lambda)$ in $\mathcal{H}$
\end{claim}
We calculate the left hand side in $\mathcal{H}[q_s^{\pm 1}]$.
We use notation in \cite[Lemma~2.10]{arXiv:1406.1003_accepted}.
Since $w^{-1}(\alpha)\in \Sigma^+\setminus\Sigma^+_Q\subset\Sigma^+\setminus\Sigma_P^+$, we have $\langle \alpha,\nu(\lambda)\rangle = \langle w^{-1}(\alpha),\nu((\lambda_P^-)^2)\rangle > 0$.
Therefore we have $\ell(\lambda n_s^{-1}) = \ell(\lambda) - 1$ by \cite[Lemma~2.17]{arXiv:1612.01312}.
Hence $q_{\lambda n_s^{-1}} = q_\lambda q_{n_s}^{-1}$.
Therefore we have
\begin{align*}
E_-(\lambda n_s^{-1}) & = E_-(n_s^{-1}(n_s\cdot \lambda))\\
& = q_{\lambda n_s^{-1}}^{1/2}q_{n_s}^{-1/2}T_{n_s^{-1}}^*\theta(n_s\cdot \lambda)\\
& = q_\lambda^{1/2}q_{n_s}^{-1}T_{n_s^{-1}}^*\theta(n_s\cdot \lambda).
\end{align*}
Hence by \cite[Lemma~2.10]{arXiv:1406.1003_accepted}, we have
\begin{equation}\label{eq:use bernstein relation}
\begin{split}
& E_-(\lambda n_s^{-1})T_{n_s}\\
& = q_\lambda^{1/2} q_{n_s}^{-1}T_{n_s^{-1}}^*\theta(n_s\cdot \lambda) T_{n_s}\\
& = q_\lambda^{1/2} q_{n_s}^{-1}T_{n_s^{-1}}^*T_{n_s}\theta(\lambda) + \sum_{k = 0}^{\langle \alpha,\nu(\lambda) - 1}q_\lambda^{1/2}q_{n_s}^{-1}T_{n_s^{-1}}^*\theta(n_s\cdot \lambda \mu_{n_s}(k))c_{n_s,k}.
\end{split}
\end{equation}
We have
\[
q_\lambda^{1/2} q_{n_s}^{-1}T_{n_s^{-1}}^*T_{n_s}\theta(\lambda)
=
q_\lambda^{1/2}\theta(\lambda) = E_{o_-}(\lambda).
\]
and if $k = 0$, since $q_{\lambda n_s^{-1}} = q_\lambda q_{n_s}^{-1}$, we have
\begin{align*}
q_\lambda^{1/2}q_{n_s}^{-1}T_{n_s^{-1}}^*\theta(n_s\cdot \lambda \mu_{n_s}(k))c_{n_s,k}.
&=
q_\lambda^{1/2}q_{n_s}^{-1}T_{n_s^{-1}}^*\theta(n_s\cdot \lambda)c_{n_s}\\
&=
q_{\lambda n_s^{-1}}^{1/2}q_{n_s}^{-1/2}T_{n_s^{-1}}^*\theta(n_s\cdot \lambda)c_{n_s}\\
&=
E_-(n_s^{-1}(n_s\cdot \lambda))c_{n_s} = E_-(\lambda n_s^{-1})c_{n_s}.
\end{align*}
We prove that if $1 \le k \le \langle \alpha,\nu(\lambda)\rangle - 1$, then $q_\lambda^{1/2}q_{n_s}^{-1}T_{n_s^{-1}}^*\theta(n_s\cdot \lambda \mu_{n_s}(k))c_{n_s,k} = 0$ in $\mathcal{H}$.
We have
\begin{align*}
&q_\lambda^{1/2}q_{n_s}^{-1}T_{n_s^{-1}}^*\theta(n_s\cdot \lambda \mu_{n_s}(k))c_{n_s,k}\\
&=
q_{\lambda}^{1/2}q_{n_s^{-1}(n_s\cdot \lambda \mu_{n_s}(k))}^{-1/2}q_{n_s}^{-1/2}E_-(n_s^{-1}(n_s\cdot \lambda \mu_{n_s}(k)))c_{n_s,k}.
\end{align*}
Hence \eqref{eq:use bernstein relation} is an expansion of $E_-(\lambda n_s^{-1})T_{n_s}$  with respect to the basis $\{E_-(w)\mid w\in W(1)\}$.
Since this is a basis of $\mathcal{H}[q_s]$ as a $C[q_s]$-module, each coefficient is in $C[q_s]$.
Hence $q_{\lambda}^{1/2}q_{n_s^{-1}(n_s\cdot \lambda \mu_{n_s}(k))}^{-1/2}q_{n_s}^{-1/2}\in C[q_s]$.
Namely, for each $s\in S_\aff$ there exists $k_s\in\Z_{\ge 0}/{\sim}$ (where the equivalence relation $\sim$ is defined by the adjoint action of $W$ on $S_\aff$) such that $q_{\lambda}^{1/2}q_{n_s^{-1}(n_s\cdot \lambda \mu_{n_s}(k))}^{-1/2}q_{n_s}^{-1/2} = \prod_{s\in S_\aff/{\sim}} q_s^{k_s}$.
We have $\sum_s k_s = (1/2)(\ell(\lambda) - \ell(n_s^{-1}(n_s\cdot \lambda \mu_{n_s}(k))) - \ell(n_s))$.
We calculate
\begin{align*}
\ell(\lambda) - \ell(n_s^{-1}(n_s\cdot \lambda \mu_{n_s}(k))) - \ell(n_s)
& \ge \ell(\lambda) - \ell(n_s^{-1}) - \ell(n_s\cdot \lambda \mu_{n_s}(k)) - \ell(n_s)\\
& = \ell(\lambda) - \ell(n_s\cdot \lambda \mu_{n_s}(k)) - 2.
\end{align*}
By \cite[Lemma~2.12]{arXiv:1406.1003_accepted}, $\ell(\lambda) - \ell(n_s\cdot \lambda \mu_{n_s}(k))\ge 2\min\{k,\langle \alpha,\nu(\lambda)\rangle - k\}$.
If the equality holds, again by \cite[Lemma~2.12]{arXiv:1406.1003_accepted}, there exists $v\in W_0$ such that $v\nu(\lambda)$ is dominant and $v(\alpha)$ is simple.

Assume that $v\nu(\lambda)$ is dominant for $v\in W_0$.
We have $\nu(\lambda) = w(\nu((\lambda_P^-)^2))$ and since $\nu((\lambda_P^-)^2)$ is dominant, we have $vw\in \Stab_{W_0}(\nu((\lambda_P^-)^2))$.
By the condition of $\lambda_P^-$, the stabilizer of $\nu((\lambda_P^-)^2)$ is $W_{0,P}$.
Hence $vw\in W_{0,P}$.
Since $w^{-1}(\alpha)\in\Sigma^+\setminus\Sigma^+_Q$, we have $w^{-1}(\alpha)\in\Sigma_{P_2}^+$.
Any element in $\Sigma_{P_2}$ is fixed by elements in $W_{0,P}$.
Hence $vw(w^{-1}(\alpha)) = w^{-1}(\alpha)$.
Therefore $v(\alpha) = w^{-1}(\alpha)$.
This is not simple by the condition on $\alpha$.

Hence we always have 
\[
\ell(\lambda) - \ell(n_s^{-1}(n_s\cdot \lambda \mu_{n_s}(k))) - \ell(n_s) > 0
\]
for $1\le k\le \langle \alpha,\nu(\lambda)\rangle - 1$.
Hence $\sum_s k_s > 0$.
Therefore there exists $s$ such that $k_s > 0$.
Hence $\prod_s q_s^{k_s} = 0$ in $\mathcal{H}$.
We get $q_\lambda^{1/2}q_{n_s}^{-1}T_{n_s^{-1}}^*\theta(n_s\cdot \lambda \mu_{n_s}(k))c_{n_s,k} = 0$.
Therefore we have
\[
E_-(\lambda n_s^{-1})T_{n_s} = E_{o_-}(\lambda) + E_-(\lambda n_s^{-1})c_{n_s}.
\]
This gives the claim.

We return to the proof of the lemma.
Since $n_w^{-1}\cdot \lambda = (\lambda_P^-)^2$ is $P$-negative, by Lemma~\ref{lem:I_P+extension as A-module}, we have
\[
\varphi(XE_{o_-}(\lambda)T_{n_w})
= \varphi(XT_{n_w})\sigma(E_{o_{-,P}}^P(n_w^{-1}\cdot \lambda)).
\]
Since $\sigma(E_{o_{-,P}}^P(n_w^{-1}\cdot \lambda)) = \sigma(E_{o_{-,P}}^P((\lambda_0^-)^2))$ is invertible, it is sufficient to prove that $\varphi(XE_{o_-}(\lambda)T_{n_w}) = 0$.
By the claim, we have
\[
\varphi(XE_{o_-}(\lambda)T_{n_w})
=
\varphi(XE_-(\lambda n_s^{-1})T_{n_s}T_{n_w})
-
\varphi(XE_-(\lambda n_s^{-1})c_{n_s}T_{n_w})
\]
By inductive hypothesis, $\varphi(XE_-(\lambda n_s^{-1})T_{n_s}T_{n_w}) = \varphi(XE_-(\lambda n_s^{-1})T_{n_{sw}}) = 0$.
We have
\[
\varphi(XE_{o_-}(\lambda n_s^{-1})c_{n_s}T_{n_w})
=
\varphi(X((\lambda n_s^{-1})\cdot c_{n_s})E_-(\lambda n_s^{-1})T_{n_w})
\]
Set $\lambda' = n_w\cdot \lambda_P^-$.
We have $\ell(\lambda n_s^{-1}) = \ell(\lambda) - 1 = \ell((\lambda')^2) - 1 = 2\ell(\lambda') - 1$ as $\langle \alpha,\nu(\lambda)\rangle > 0$.
Since $\langle \alpha,\nu(\lambda')\rangle = \langle w^{-1}(\alpha),\nu(\lambda_P^-)\rangle > 0$, we have $\ell(\lambda') - 1 = \ell(\lambda'n_s^{-1})$.
By \cite[Lemma~2.15]{arXiv:1612.01312}, we have $\ell(\lambda') = \ell(n_s\cdot \lambda')$.
Hence $\ell(\lambda n_s^{-1}) = \ell(\lambda'n_s^{-1}) + \ell(n_s\cdot \lambda')$.
Therefore
\[
E_-(\lambda n_s^{-1})
=
E_-(\lambda' n_s^{-1}(n_s\cdot \lambda'))
=
E_-(\lambda' n_s^{-1})E_{o_-}(n_s\cdot \lambda')
\]
by the definition of $E_-(\lambda n_s^{-1})$.
Hence
\begin{align*}
& \varphi(X((\lambda n_s^{-1})\cdot c_{n_s})E_-(\lambda n_s^{-1})T_{n_w})\\
& =
\varphi(X((\lambda n_s^{-1})\cdot c_{n_s})E_-(\lambda' n_s^{-1})E_{o_-}(n_s\cdot \lambda')T_{n_w}).
\end{align*}
Since $w^{-1}(\alpha) \in \Sigma^+\setminus\Sigma_Q^+\subset\Sigma^+\setminus\Sigma_P^+$, we have $\langle w^{-1}(\alpha),\nu(n_w^{-1}n_s\cdot \lambda')\rangle = -\langle \alpha,\nu(\lambda')\rangle = -\langle w^{-1}(\alpha),\nu(\lambda_P^-)\rangle < 0$.
Therefore $n_w^{-1}n_s\cdot \lambda'$ is not $Q$-negative.
Hence $\varphi(X((\lambda n_s^{-1})\cdot c_{n_s})E_-(\lambda' n_s^{-1})E_{o_-}(n_s\cdot \lambda')T_{n_w}) = 0$ by Proposition~\ref{prop:I_P as A-module}.
\end{proof}

\subsection{The homomorphism $I'_Q(e_Q(\sigma))\to I'_{Q_1}(e_{Q_1}(\sigma))$}
Let $Q_1\supset Q \supset P$ be parabolic subgroups.
Recall that we have the homomorphism $I'_Q(e_Q(\sigma))\to I'_{Q_1}(e_{Q_1}(\sigma))$.
This is defined by $I_Q(\St_P^Q(\sigma^{\iota_P}_{\ell - \ell_P}))\to I_{Q_1}(\St_P^{Q_1}(\sigma^{\iota_P}_{\ell - \ell_P}))$ with $\iota$.
We give the following description of this homomorphism.
\begin{prop}\label{prop:explicit formula of I'_Q(e_Q) -> I'_(Q_1)(e_(Q_1))}
Let $\varphi\in I'_Q(e_Q(\sigma))$ and $\varphi'\in I'_{Q_1}(e_{Q_1}(\sigma))$ be the image of $\varphi$.
Then for $w\in W_0^{Q_1}$, we have $\varphi'(T_{n_w}^*) = (-1)^{\ell(w_{Q_1}w_Q)}\varphi(T_{n_{ww_{Q_1}w_Q}}^*)$.
In particular, combining with \cite[Proposition~4.12]{arXiv:1612.01312}, we have
\begin{align*}
&\Ker(I'_Q(e_Q(\sigma))\to I'_{Q_1}(e_{Q_1}(\sigma))) \\
&= \{\varphi\in I'_Q(e_Q(\sigma))\mid \text{$\varphi(T_{n_w}^*) = 0$ for any $w\in W_0^{Q_1}w_{Q_1}w_Q$}\}.
\end{align*}
\end{prop}
First we describe the homomorphism $I_Q(\St_P^Q(\sigma))\to I_{Q_1}(\St_P^{Q_1}(\sigma))$.
Recall that the kernel of $I_P^Q(\sigma)\ni\varphi\mapsto \varphi(T_{n_{w_Qw_P}})\in \sigma$ is $\sum_{Q\supset P_1\supsetneq P}I_{P_1}(e_{P_1}(\sigma))$ and hence it gives an identification $\sigma\simeq \St_P^Q(\sigma)$ as vector spaces by Lemma~\ref{lem:embedding of smaller parabolic induction}.
\begin{lem}
The homomorphism $I_Q(\St_P^Q(\sigma))\to I_{Q_1}(\St_P^{Q_1}(\sigma))$ is given by $\varphi\mapsto (X\mapsto \varphi(XT_{n_{w_{Q_1}w_Q}}))$. (Here we identify $\St_P^{Q}(\sigma)$ and $\St_P^{Q_1}(\sigma)$ with $\sigma$.)
\end{lem}
\begin{proof}
Since $I_P^Q(\sigma)\to \St_P^Q(\sigma)$ is given by $\varphi\mapsto \varphi(T^Q_{n_{w_Qw_P}})$ (under the identification $\sigma = \St_P^Q(\sigma)$), $I_P(\sigma)\to I_Q(\St_P^Q(\sigma))$ is given by $\varphi\mapsto (X\mapsto \varphi(XT_{n_{w_Qw_P}}))$.
Now recall the following commutative diagram which defines the homomorphism in the lemma.
\[
\begin{tikzcd}
I_P(\sigma)\arrow[d,equal]\arrow[r] & I_Q(\St_P^Q(\sigma))\arrow[d]\\
I_P(\sigma)\arrow[r] & I_{Q_1}(\St_P^{Q_1}(\sigma)).
\end{tikzcd}
\]
Let $\varphi\in I_Q(\St_P^Q(\sigma))$ and take $\widetilde{\varphi}\in I_P(\sigma)$ which is a lift of $\varphi$.
Then we have $\varphi(X) = \widetilde{\varphi}(XT_{n_{w_Qw_P}})$.
Let $\varphi'$ be the image of $\varphi$.
Then the above commutative diagram we have $\varphi'(X) = \widetilde{\varphi}(XT_{n_{w_{Q_1}w_P}})$.
Since $n_{w_{Q_1}w_P} = n_{w_{Q_1}w_Q}n_{w_Qw_P}$, we have $\widetilde{\varphi}(XT_{n_{w_{Q_1}w_P}}) = \widetilde{\varphi}(XT_{n_{w_{Q_1}w_Q}}T_{n_{w_{Q}w_P}}) = \varphi(XT_{n_{w_{Q_1}w_Q}})$.
\end{proof}

\begin{proof}[Proof of Proposition~\ref{prop:explicit formula of I'_Q(e_Q) -> I'_(Q_1)(e_(Q_1))}]
Let $\varphi\in I'_Q(e_Q(\sigma))$ and $\varphi'\in I'_{Q_1}(e_{Q_1}(\sigma))$ its image.
Then $\varphi\circ\iota\in I_Q(\St^Q_P(\sigma^{\iota_P}_{\ell - \ell_P}))$ and $\varphi'\circ\iota\in I_{Q_1}(\St^{Q_1}_P(\sigma^{\iota_P}_{\ell - \ell_P}))$.
By the above lemma, we have $\varphi'\circ \iota(X) = \varphi\circ\iota(XT_{n_{w_{Q_1}w_Q}})$.
Hence $\varphi'(\iota(X)) = \varphi(\iota(XT_{n_{w_{Q_1}w_Q}})) = (-1)^{\ell(w_{Q_1}w_Q)}\varphi(\iota(X)T_{n_{w_{Q_1}w_Q}}^*)$ for any $X\in \mathcal{H}$.
Therefore, for any $X\in \mathcal{H}$, we have $\varphi'(X) = (-1)^{\ell(w_{Q_1}w_Q)}\varphi(XT_{n_{w_{Q_1}w_Q}}^*)$ for any $X\in \mathcal{H}$.
\end{proof}
By Proposition~\ref{prop:explicit formula of I'_Q(e_Q) -> I'_(Q_1)(e_(Q_1))},  $\varphi\in \Ker(I'_Q(e_Q(\sigma))\to \bigoplus_{Q_1\supsetneq Q}I'_{Q_1}(e_{Q_1}(\sigma)))$ if and only if $\varphi(T^*_{n_w}) = 0$ for any $w\in \bigcup_{Q_1\supsetneq Q}W_0^{Q_1}w_{Q_1}w_Q$.
We get the following description of the kernel appearing in Lemma~\ref{lem:exact sequence related to I and I'}.
\begin{lem}\label{lem:kernel of I'_Q(e_Q) -> plusI'_(Q_1)(e_(Q_1))}
Let $w\in W_0^Q$.
We have $\Delta_{ww_Q} \ne \Delta\setminus\Delta_Q$ if and only if for some $Q_1\supsetneq Q$, $w\in W_0^{Q_1}w_{Q_1}w_Q$.
Hence we have
\begin{align*}
&\Ker\left(I'_Q(e_Q(\sigma))\to \bigoplus_{Q_1\supsetneq Q}I'_{Q_1}(e_{Q_1}(\sigma))\right)\\
&=
\{\varphi\in I'_Q(e_Q(\sigma))\mid \text{$\varphi(T_{n_w}^*) = 0$ for any $w\in W_0^Q$ such that $\Delta_{ww_Q}\ne \Delta\setminus\Delta_Q$}\}.
\end{align*}
\end{lem}
\begin{proof}
Let $w\in W_0^Q$ and assume that for some $Q_1\supsetneq Q$ and $v\in W_0^{Q_1}$, we have $w = vw_{Q_1}w_Q$.
Then for $\alpha\in\Delta_{Q_1}\setminus\Delta_Q$, $w_{Q_1}(\alpha) \in \Sigma_{Q_1}^-$.
Since $v\in W_0^{Q_1}$, we have $vw_{Q_1}(\alpha) < 0$.
Hence $ww_Q(\alpha) < 0$.
Therefore $\alpha\notin \Delta_{ww_Q}$.
Hence $\Delta_{ww_Q}\ne \Delta\setminus\Delta_Q$.

Assume that $\Delta_{ww_Q}\ne\Delta\setminus\Delta_Q$.
Since $w\in W_0^Q$, for any $\alpha\in\Delta_Q$ we have $ww_Q(\alpha) < 0$.
Hence $\alpha\notin \Delta_{ww_Q}$.
Therefore $\Delta_Q\subset\Delta\setminus\Delta_{ww_Q}$, namely we have $\Delta_{ww_Q}\subset\Delta\setminus\Delta_Q$.
Hence $\Delta_{ww_Q}\not\supset \Delta\setminus\Delta_Q$.
Take $\alpha\in(\Delta\setminus\Delta_Q)\setminus\Delta_{ww_Q}$.
Let $Q_1$ be a parabolic subgroup corresponding to $\Delta_Q\cup\{\alpha\}$.
Then we have $\Delta_{Q_1}\subset\Delta\setminus\Delta_{ww_Q}$.
If $\beta\in\Delta_{Q_1}$, then $w_{Q_1}(\beta)\in-\Delta_{Q_1}\subset-(\Delta\setminus\Delta_{ww_Q})$.
Hence $ww_Qw_{Q_1}(\beta) > 0$.
Therefore $ww_Qw_{Q_1}\in W^{Q_1}_0$.
We have $w\in W_0^{Q_1}w_{Q_1}w_Q$.
\end{proof}

\subsection{Complex}
In this subsection, we prove that the sequence in Lemma~\ref{lem:exact sequence related to I and I'} is a complex, namely the composition
\[
I_Q(e_Q(\sigma))\to I'_Q(e_Q(\sigma))\to I'_{Q_1}(e_{Q_1}(\sigma))
\]
is zero for any parabolic subgroup $Q_1\supsetneq Q$.

We have the following diagram
\[
\begin{tikzcd}
I_Q(e_Q(\sigma)) \arrow[rr]\arrow[dash,d,"\wr"] &[-.3cm] &[-.3cm] I'_Q(e_Q(\sigma))\arrow[r]\arrow[d,dash,"\wr"] & I'_{Q_1}(e_{Q_1}(\sigma))\\
I_{Q_1}I_Q^{Q_1}(e_Q(\sigma))\arrow[r] & I'_{Q_1}(I_{Q}^{Q_1}(e_{Q}(\sigma))) \arrow[r] & I'_{Q_1}(I_Q^{Q_1\prime}(e_Q(\sigma))).\arrow[ru]
\end{tikzcd}
\]
This is commutative by Lemma~\ref{lem:transitivity of I->I'}.
The sequence
\[
I'_{Q_1}(I_{Q}^{Q_1}(e_{Q}(\sigma))) \to I'_{Q_1}(I_Q^{Q_1\prime}(e_Q(\sigma)))\to I'_{Q_1}(e_{Q_1}(\sigma))
\]
comes from
\[
I_{Q}^{Q_1}(e_{Q}(\sigma)) \to I_Q^{Q_1\prime}(e_Q(\sigma))\to e_{Q_1}(\sigma).
\]
Since $R_{Q}^{Q_1}(e_{Q_1}(\sigma)) = 0$ by \cite[Lemma~5.17]{arXiv:1612.01312}, this composition is zero.

\subsection{Exactness}\label{subsec:exactness}
Now we finish the proof of Lemma~\ref{lem:exact sequence related to I and I'} by proving the following lemma.
\begin{lem}\label{lem:detect the kernel}
We have
\[
\Ima\left(I_Q(e_Q(\sigma))\to I'_Q(e_Q(\sigma))\right)\supset\Ker\left(I'_Q(e_Q(\sigma))\to \bigoplus_{Q_1\supsetneq Q}I'_{Q_1}(e_{Q_1}(\sigma))\right).
\]
\end{lem}
We start with the following lemma.
\begin{lem}
Let $\psi\in I'_Q(e_Q(\sigma))$ and $w\in W_0^Q$.
Then we have $\psi(T_{n_w}^*) = \sum_{v\in W_0^Q,v\le w}\psi(T_{n_v})$.
\end{lem}
\begin{proof}
The same argument as the proof of \cite[IV.9 Proposition]{MR3600042} implies
\[
\psi(T_{n_w}^*) = \sum_{v\le w}\psi(T_{n_v}).
\]
If $v\notin W_0^Q$, then there exists $v_1\in W_0^Q$ and $v_2\in W_{0,Q}\setminus\{1\}$ such that $v = v_1v_2$.
We have $\ell(v_1v_2) = \ell(v_1) + \ell(v_2)$.
Hence $\psi(T_{n_v}) = \psi(T_{n_{v_1}}T_{n_{v_2}}) = \psi(T_{n_{v_1}})e_Q(\sigma)(T^Q_{n_{v_2}}) = 0$ by the definition of $e_Q(\sigma)$.
We get the lemma.
\end{proof}
Let $\mu^Q$ be the M\"obius function associated to $(W_0^Q,\le)$.
Then by the above lemma, we have
\[
\psi(T_{n_w}) = \sum_{v\le w,v\in W_0^Q}\mu^Q(v,w)\psi(T_{n_v}^*).
\]
By \cite[Theorem~1.2]{MR0435249}, $\mu^Q$ is given by
\[
\mu^Q(v,w) = 
\begin{cases}
0 & \text{(there exists $\alpha\in\Delta_Q$ such that $vs_\alpha\le w$)},\\
(-1)^{\ell(v) + \ell(w)} & \text{(otherwise)}.
\end{cases}
\]

Now we assume that $\psi\in \Ker(I'_Q(e_Q(\sigma))\to \bigoplus_{Q_1\supsetneq Q}I'_{Q_1}(e_{Q_1}(\sigma)))$.
Then we have $\psi(T_{n_w}^*) = 0$ if $\Delta_{ww_Q} \ne \Delta\setminus\Delta_Q$ by Lemma~\ref{lem:kernel of I'_Q(e_Q) -> plusI'_(Q_1)(e_(Q_1))}.
Hence we have $\psi(T_{n_w}) = \sum_{v\le w,v\in W_0^Q,\Delta_{vw_Q} = \Delta\setminus\Delta_Q}\mu^Q(v,w)\psi(T_{n_{v}}^*)$.
Let $w_c$ be the longest element of the finite Weyl group of the parabolic subgroup corresponding to $\Delta\setminus\Delta_Q$.
Then $\Delta_{w_Gw_c} = \Delta\setminus\Delta_Q$ and $w = w_Gw_c$ is maximal in $\{w\in W_0\mid \Delta_w = \Delta\setminus\Delta_Q\}$ \cite[Remark~3.16]{arXiv:1406.1003_accepted}.

\begin{lem}\label{lem:calculation of Mobius function}
Let $w\in W_0^Q$ such that $\Delta_{ww_Q} = \Delta\setminus\Delta_Q$.
Then we have
\[
\mu^Q(w,w_Gw_Q) = 
\begin{cases}
(-1)^{\ell(w_c)} & (w = w_Gw_cw_Q),\\
0 & (w\ne w_Gw_cw_Q).
\end{cases}
\]
\end{lem}
We prove this lemma by backward induction on the length of $w$.
For the inductive step, we use the following.
\begin{lem}\label{lem:take alpha and the same Delta, some conclusion}
Let $w\in W_0^Q$ such that $\Delta_{ww_Q} =\Delta\setminus\Delta_Q$, $w\ne w_Gw_cw_Q$ and $\alpha\in\Delta$ such that $s_\alpha ww_Q > ww_Q$ and $\Delta_{ww_Q} = \Delta_{s_\alpha ww_Q}$. (Such $\alpha$ exists~\cite[Lemma~3.15]{arXiv:1406.1003_accepted}.)
We have
\begin{enumerate}
\item $s_\alpha w \in W_0^Q$.
\item $s_\alpha w > w$.
\end{enumerate}
\end{lem}
\begin{proof}
If $\beta\in\Delta_Q$, then $w_Q(\beta) \in -\Delta_Q$.
Hence $s_\alpha ww_Q(w_Q(\beta)) > 0$ since $\Delta_{s_\alpha ww_Q} = \Delta\setminus\Delta_Q$.
Therefore $s_\alpha w(\beta) > 0$ for any $\beta\in\Delta_Q$.
Namely we have $s_\alpha w\in W_0^Q$.
Therefore we have $s_\alpha w,w\in W_0^Q$, $w_Q\in W_{0,Q}$ and $s_\alpha ww_Q > ww_Q$.
By \cite[Lemma~3.5]{MR0435249}, $s_\alpha w > w$.
\end{proof}

\begin{proof}[Proof of Lemma~\ref{lem:calculation of Mobius function}]
Assume that $w = w_Gw_cw_Q$ and there exists $\alpha\in \Delta_Q$ such that $ws\le w_Gw_Q$ where $s = s_\alpha$.
Then we have $w_Gw_cw_Qs\le w_Gw_Q$.
Hence $w_cw_Qs\ge w_Q$.
Let $Q_0$ be the parabolic subgroup corresponding to $\Delta\setminus\Delta_Q$.
Then $w_c\in W_{0,Q_0}$ and $W_{0,Q}\subset {}^{Q_0}W_0$.
Therefore $w_Q,w_Qs\in {}^{Q_0}W_0$.
Hence by \cite[Lemma~3.5]{MR0435249}, we have $w_Qs\ge w_Q$.
This is a contradiction.
Hence $\mu^Q(w,w_Gw_Q) = (-1)^{\ell(w) + \ell(w_Gw_Q)} = (-1)^{\ell(w_c)}$.

If $w\ne w_Gw_cw_Q$, then $ww_Q \ne w_Gw_c$.
Take $\alpha$ as in the previous lemma.
Assume that $s_\alpha w = w_Gw_cw_Q$.
Since $s_\alpha w_Gw_c = ww_Q < s_\alpha ww_Q = w_Gw_c$, we have $(w_Gw_c)^{-1}(\alpha) < 0$.
Put $\alpha' = -w_G^{-1}(\alpha)\in\Delta$.
Then we have $w_c^{-1}(\alpha') > 0$.
Hence $\alpha'\in\Delta_Q$.
Put $\beta = -w_Q(\alpha')\in \Delta_Q$ and we prove $ws_\beta\le w_Gw_Q$.
We have $ws_\beta = s_\alpha w_Gw_cw_Q s_\beta = w_G s_{\alpha'}w_cw_Qs_\beta$.
Hence it is sufficient to prove that $s_{\alpha'}w_cw_Qs_\beta\ge w_Q$.

We have $\Delta_{w_Gs_{\alpha'}w_c} = \Delta_{s_\alpha w_Gw_c} = \Delta_{ww_Q} = \Delta\setminus\Delta_Q$.
Hence $\Delta_{s_{\alpha'}w_c} = \Delta_Q$.
In particular, $s_{\alpha'}w_c\in W_0^Q$.
Hence $\ell(s_{\alpha'}w_cw_Qs_\beta) = \ell(s_{\alpha'}w_c) + \ell(w_Qs_\beta)$ as $w_Q,s_\beta\in W_{0,Q}$.
Since $\alpha'\in \Delta_Q$, $s_{\alpha'}\in W_{0,Q}\subset W_0^{Q_0}$.
Hence we have $\ell(s_{\alpha'}w_c) = \ell(s_{\alpha'}) + \ell(w_c)$.
Therefore we get
\[
\ell(s_{\alpha'}w_cw_Qs_\beta) = \ell(s_{\alpha'}) + \ell(w_c) + \ell(w_Qs_\beta).
\]
Hence $s_{\alpha'}w_cw_Qs_\beta\ge s_{\alpha'}w_Qs_\beta = w_Q$.

Finally assume that $s_\alpha w \ne w_Gw_c$ and we prove the lemma by backward induction on $\ell(ww_Q)$.
By inductive hypothesis, there exists $\beta\in\Delta_Q$ such that $s_\alpha ws_\beta\le w_Gw_Q$.
Since $s_\alpha w\in W_0^Q$, we have $s_\alpha ws_\beta  > s_\alpha w$.
Therefore we have $s_\alpha ws_\beta > s_\alpha w > w$.
By Property $Z(ws_\beta,s_\alpha ws_\beta,s_\beta)$ \cite[\S 1, Remarks (2)]{MR0435249}, we have $ws_\beta \le s_\alpha ws_\beta$.
Since we have $s_\alpha ws_\beta\le w_Gw_Q$, we get $ws_\beta \le w_Gw_Q$.
\end{proof}
Therefore we get
\begin{lem}
If $\psi\in \Ker(I'_Q(e_Q(\sigma))\to \bigoplus_{Q_1\supsetneq Q}I'_{Q_1}(e_{Q_1}(\sigma)))$, then $\psi(T_{n_{w_Gw_Q}}) = (-1)^{\ell(w_c)}\psi(T_{n_{w_Gw_cw_Q}}^*)$.
\end{lem}
Since the image of $I_Q(e_Q(\sigma))\to I'_Q(e_Q(\sigma))$ is contained in the kernel by the previous subsection, we get the following lemma.
\begin{lem}
Consider a linear map $I'_Q(e_Q(\sigma))\to \sigma$ defined by $\psi\mapsto \psi(T_{n_{w_Gw_cw_Q}}^*)$.
Then the composition $I_Q(e_Q(\sigma))\to I'_Q(e_Q(\sigma))\to \sigma$ is surjective.
\end{lem}
\begin{proof}
Let $\psi\in I'_Q(e_Q(\sigma))$ be the image of $\varphi\in I_Q(e_Q(\sigma))$.
Then the characterization of the homomorphism (Proposition~\ref{prop:hom between I and I'}) gives $\varphi(T_{n_{w_Gw_Q}}) = \psi(T_{n_{w_Gw_Q}})$.
The lemma follows from the previous lemma and the surectivity of the map $I_Q(e_Q(\sigma))\ni \varphi\mapsto \varphi(T_{n_{w_Gw_Q}})\in \sigma$.
\end{proof}

The following lemma ends the proof of Lemma~\ref{lem:detect the kernel}, hence that of Theorem~\ref{thm:twist of Steinberg}.
\begin{lem}\label{lem:final part of the exactness}
For $w\in W_0^Q$ such that $\Delta_{ww_Q} = \Delta\setminus\Delta_Q$ and $x\in\sigma$, there exists $\psi\in \Ima(I_Q(e_Q(\sigma))\to I'_Q(e_Q(\sigma)))$ such that for any $v\in W_0^Q$ we have
\begin{align*}
\psi(T_{n_v}^*) & = 
\begin{cases}
x & (v = w),\\
0 & (v\ne w).
\end{cases}
\end{align*}
\end{lem}
We need one lemma.
\begin{lem}
Let $w\in W_0^P$ and $\lambda\in \Lambda(1)$.
Then for $\varphi\in I'_P(\sigma)$, we have
\[
(\varphi E_{o_+}(\lambda))(T^*_{n_w})
=
\begin{cases}
\varphi(T^*_{n_w})\sigma(E_{o_{+,P}}^P(n_w^{-1}\cdot \lambda)) & (n_w^{-1}\cdot \lambda\in W_P^-(\lambda)),\\
0 & (n_w^{-1}\cdot \lambda\notin W_P^-(\lambda)).
\end{cases}
\]
\end{lem}
\begin{proof}
Set $\varphi^\iota = \varphi\circ \iota$ and $\sigma' = \sigma^{\iota_P}_{\ell - \ell_P}$.
Then we have $\varphi^\iota\in I_P(\sigma')$ \cite[Proposition~4.11]{arXiv:1612.01312}.
Hence, by Proposition~\ref{prop:I_P as A-module}, we have
\[
\varphi^\iota(E_{o_-}(\lambda)T_{n_w}) =
\begin{cases}
 \varphi^\iota(T_{n_w})\sigma'(E^P_{o_-}(n_w^{-1}\cdot \lambda))& (n_w^{-1}\cdot \lambda\in W_P^-(\lambda)),\\
0 & (n_w^{-1}\cdot \lambda\notin W_P^-(\lambda)).
\end{cases}
\]
The left hand side is
\[
\varphi^\iota(E_{o_-}(\lambda)T_{n_w}) =
\varphi(\iota(E_{o_-}(\lambda))\iota(T_{n_w})) =
\varphi((-1)^{\ell(\lambda)}E_{o_+}(\lambda)(-1)^{\ell(n_w)}T_{n_w}^*)
\]
by \cite[Lemma~5.31]{MR3484112}.
Therefore if $n_w^{-1}\cdot \lambda\notin W_P^-(\lambda)$, then $\varphi(E_{o_+}(\lambda)T_{n_w}^*) = 0$.

If $n_w^{-1}\cdot \lambda\in W_P^-(\lambda)$, then
\begin{align*}
\sigma'(E^P_{o_-}(n_w^{-1}\cdot \lambda))
& = 
(-1)^{\ell(n_w^{-1}\cdot \lambda) - \ell_P(n_w^{-1}\cdot \lambda)}\sigma(\iota_P(E_{o_-}^P(n_w^{-1}\cdot \lambda)))\\
& =
(-1)^{\ell(n_w^{-1}\cdot \lambda)}\sigma(E_{o_+}^P(n_w^{-1}\cdot \lambda))
\end{align*}
again by \cite[Lemma~5.31]{MR3484112} and \cite[Lemma~4.5]{arXiv:1612.01312}.
We also have $\varphi^\iota(T_{n_w}) = (-1)^{\ell(n_w)}\varphi(T_{n_w}^*)$.
Hence we get
\[
\varphi((-1)^{\ell(\lambda)}E_{o_+}(\lambda)(-1)^{\ell(n_w)}T_{n_w}^*)
=
(-1)^{\ell(n_w)}(-1)^{\ell(n_w^{-1}\cdot \lambda)}\varphi(T_{n_w}^*)\sigma(E_{o_+}^P(n_w^{-1}\cdot \lambda))
\]
Since $\ell(n_w^{-1}\cdot \lambda) = \ell(\lambda)$~\cite[Lemma~2.15]{arXiv:1612.01312}, we get the lemma.
\end{proof}

\begin{proof}[Proof of Lemma~\ref{lem:final part of the exactness}]
We prove the lemma by induction on the length of $ww_Q$.
Assume that $ww_Q = w_Gw_c$, namely $w = w_Gw_cw_Q$.
Let $\lambda_P^-\in Z(W_P(1))$ as in Proposition~\ref{prop:localization as Levi subalgebra}.
Notice that $n_{w_Q}^{-1}\cdot \lambda_P^-$ is $P$-negative since $w_Q$ (in fact, any element in $W_0$) preserves $\Sigma^+\setminus\Sigma_P^+ = \Sigma_{P_2}^+$.
We also have that $n_{w_Q}^{-1}\cdot \lambda_P^-\in Z(W_P(1))$ since $n_{w_Q}$ normalizes $W_P(1)$.
Hence $e_Q(\sigma)(T_{n_{w_Q}^{-1}\cdot \lambda_P^-}^{Q*}) = \sigma(T_{n_{w_Q}^{-1}\cdot \lambda_P^-}^{P*})$ is invertible.
Take $\psi_0 \in \Ima(I_Q(e_Q(\sigma))\to I'_Q(e_Q(\sigma)))$ such that $\psi_0(T^*_{n_{w_Gw_cw_Q}}) = xe_Q(\sigma)(T_{n_{w_Q}^{-1}\cdot \lambda_P^-}^{Q*})^{-1}$.
Put $\lambda = n_{w_Gw_c}\cdot \lambda_P^-$ and set $\psi = \psi_0 E_{o_+}(\lambda)$.
Let $v\in W_0^Q$.
If $v\ne w_Gw_cw_Q$, then since $v$ and $w_Gw_cw_Q$ are in $W^Q_0$, we have $v\notin w_Gw_cw_QW_{0,Q}$.
Hence $(w_Gw_c)^{-1}v\notin W_{0,Q}$.
Therefore there exists $\alpha\in\Sigma^+\setminus\Sigma_Q^+$ such that $(w_Gw_c)^{-1}v(\alpha) < 0$.
Since $\Sigma^+\setminus\Sigma_Q^+\subset \Sigma_{P_2}$ and, $\Sigma_{P_2}$ is stabilized by $W_0$, we have $(w_Gw_c)^{-1}v(\alpha)\in \Sigma_{P_2}^- = \Sigma^-\setminus\Sigma_P^-$.
Hence $\langle (w_Gw_c)^{-1}v(\alpha),\nu(\lambda_P^-)\rangle < 0$.
The left hand side is $\langle \alpha,\nu(n_v^{-1}\cdot \lambda)\rangle$.
Hence $n_v^{-1}\cdot \lambda$ is not $Q$-negative.
Therefore $\psi(T_{n_v}) = (\psi_0 E_{o_+}(\lambda))(T_{n_v}) = 0$ if $v\ne w_Gw_cw_Q$ by the above lemma.

Assume that $v = w_Gw_cw_Q$.
Then $n_v^{-1}\cdot \lambda = n_{w_Q}^{-1}\cdot \lambda_P^-$.
Since $\lambda_P^-$ is dominant, it is $Q$-negative.
The set of $Q$-negative elements is stable under the conjugate action of $W_{Q}(1)$.
Hence $n_{w_Q}^{-1}\cdot \lambda_P^-$ is also $Q$-negative.
Therefore we have $\psi(T^*_{n_{w_Gw_cw_Q}}) = (\psi_0 E_{o_+}(\lambda))(T^*_{n_{w_Gw_cw_Q}}) = \psi_0(T^*_{n_{w_Gw_cw_Q}})e_Q(\sigma)(T^{Q*}_{n_{w_Q}^{-1}\cdot \lambda_P^-}) = x$ by the previous lemma.
We get the lemma when $w = w_Gw_cw_Q$.

Assume that $w \ne w_Gw_cw_Q$ and take $\alpha$ such that $s_\alpha ww_Q > ww_Q$ and $\Delta_{s_\alpha ww_Q} = \Delta_{ww_Q}$ as in Lemma~\ref{lem:take alpha and the same Delta, some conclusion}.
Then $s_\alpha w > w$ by Lemma~\ref{lem:take alpha and the same Delta, some conclusion}.
Set $s = s_\alpha$.
By inductive hypothesis, there exists $\psi_0\in \Ima(I_Q(e_Q(\sigma))\to I'_Q(e_Q(\sigma)))$ such that for $v\in W^Q_0\setminus\{s_\alpha w\}$, $\psi_0(T_{n_v}^*) = 0$ and $\psi_0(T_{n_{s_\alpha w}}^*) = x$.
We prove that $\psi = \psi_0T_{n_s}$ satisfies the condition of the lemma.
First we have
\[
\psi(T_{n_w}^*) = \psi_0(T_{n_s}T_{n_w}^*) = \psi_0((T_{n_s}^* - c_{n_s})T_{n_w}^*) = \psi_0(T_{n_s}^*T_{n_w}^*) - \psi_0(c_{n_s}T_{n_w}^*).
\]
Since $sw > w$, we have $T_{n_s}^*T_{n_w}^* = T_{n_{sw}}^*$.
Hence $\psi_0(T_{n_s}^*T_{n_w}^*) = x$.
Since $\psi_0(T_{n_w}^*) = 0$, we have $\psi_0(c_{n_s}T_{n_w}^*) = \psi_0(T_{n_w}^*)e_Q(\sigma)(n_w^{-1}\cdot c_{n_s}) = 0$.
Hence $\psi(T_{n_w}^*) = x$.

Assume that $v\ne w$.
If $\ell(s) + \ell(v) > \ell(sv)$, then $T_{n_s}T_{n_v}^* = T_{n_s}T_{n_s}^*T_{n_{sv}}^* = 0$.
Hence $\psi(T_{n_v}^*) = 0$.
If $\ell(s) + \ell(v) = \ell(sv)$, then
\[
\psi(T_{n_v}^*) = \psi_0(T_{n_s}T_{n_v}^*) = \psi_0((T_{n_s}^* - c_{n_s})T_{n_v}^*) = \psi_0(T^*_{n_{sv}}) - \psi_0(c_{n_s}T_{n_v}^*).
\]
Since $v\ne w$, $sv \ne sw$.
Hence $\psi_0(T^*_{n_{sv}}) = 0$.
Since $sv > v$ and $s(sw) < sw$, we have $v\ne sw$.
Hence $\psi_0(c_{n_s}T_{n_v}^*) = \psi_0(T_{n_v}^*)e_Q(\sigma)(n_v^{-1}\cdot c_{n_s}) = 0$.
Therefore we get $\psi(T_{n_v}^*) = 0$.
\end{proof}

\subsection{Supersingular modules}
\begin{prop}
If $\pi$ is supsersingular, then $\pi^\iota$ is also supsersingular.
\end{prop}
\begin{proof}
For each $W(1)$-orbit $\mathcal{O}\subset \Lambda(1)$ such that $\ell(\mathcal{O}) > 0$, we have
\[
\iota(z_\mathcal{O}) = \sum_{\lambda\in \mathcal{O}}\iota(E_{o_-}(\lambda)) = \sum_{\lambda\in \mathcal{O}}(-1)^{\ell(\mathcal{O})}E_{o_+}(\lambda) = (-1)^{\ell(\mathcal{O})}z_\mathcal{O}
\]
by \cite[Lemma~5.31]{MR3484112}.
Hence $\pi(z_\mathcal{O}^n) = 0$ implies $\pi^\iota(z_\mathcal{O}^n) = 0$.
\end{proof}

\begin{prop}\label{prop:twist of supersingular}
Assume that $C$ is a field.
Let $\pi = \pi_{\chi,J,V}$ be a simple supersingular representation.
Then $\pi^\iota = \pi_{\chi,S_{\aff,\chi}\setminus J,V}$.
\end{prop}
\begin{proof}
Let $\Xi = \Xi_{\chi,J}$ be the character of $\mathcal{H}_\aff$ parametrized by $(\chi,J)$.
The representation $\pi$ is given by $\pi = (V\otimes\Xi_{\chi,J})\otimes_{\mathcal{H}_\aff C[\Omega(1)_\Xi]}\mathcal{H}$.
The homomorphism $\iota$ preserves $\mathcal{H}_\aff$ and $C[\Omega(1)_\Xi]$. (On $C[\Omega(1)]$, $\iota$ is identity.)
Hence we get $\pi^\iota = (V^\iota\otimes\Xi^\iota)\otimes_{\mathcal{H}_\aff C[\Omega(1)_\Xi]}\otimes\mathcal{H}$.
Since $\iota$ is trivial on $C[\Omega(1)_\Xi]$, $V^\iota = V$.
Let $(\chi',J')$ be the pair such that $\Xi^\iota$ is parametrized by $(\chi',J')$.
The character $\chi'$ is a direct summand of $V^\iota|_{Z_\kappa\cap W_\aff(1)}$ and since $V^\iota = V$, we have $V^\iota|_{Z_\kappa\cap W_\aff(1)} = V|_{Z_\kappa\cap W_\aff(1)}$.
Since $V|_{Z_\kappa\cap W_\aff(1)}$ is a direct sum of $\chi$, $\chi' = \chi$.
The subset $J\subset S_{\aff,\chi}$ satisfies
\[
\Xi(T_{\widetilde{s}}) =
\begin{cases}
0 & (s\in J),\\
\chi(c_{\widetilde{s}}) & (s\in S_{\aff,\chi}\setminus J).
\end{cases}
\]
where $\widetilde{s}\in W_\aff(1)$ is a lift of $s$.
We have $\Xi(\iota(T_{\widetilde{s}})) = -\Xi(T_{\widetilde{s}} - c_{\widetilde{s}}) = -\Xi(T_{\widetilde{s}}) + \chi(c_{\widetilde{s}})$.
Therefore we have
\[
\Xi^\iota(T_{\widetilde{s}}) =
\begin{cases}
\chi(c_{\widetilde{s}}) & (s\in J),\\
0 & (s\in S_{\aff,\chi}\setminus J).
\end{cases}
\]
We have $J' = \{s\in S_{\aff,\chi}\mid \Xi^\iota(T_{n_s}) = 0\}$.
Hence $J' = S_{\aff,\chi}\setminus J$.
\end{proof}

\subsection{Simple modules}
Assume that $C$ is an algebraically closed field.
Summarizing the results in this section, we have the following.
We need notation.
By \cite[Remark~4.6]{arXiv:1612.01312}, $T_w^P\mapsto (-1)^{\ell(w) - \ell_P(w)}T_w^{P}$ is an algebra homomorphism of $\mathcal{H}_P$.
This preserves the subalgebra $C[\Omega_P(1)]$.
Let $\Xi$ be a character of $\mathcal{H}_{\aff,P}$.
Then the above homomorphism also preserves $C[\Omega_P(1)_\Xi]$ since the homomorphism is trivial on $\mathcal{H}_{\aff,P}$ by \cite[Lemma~4.7]{arXiv:1612.01312}.
For a $C[\Omega_P(1)_\Xi]$-module $V$, let $V_{\ell - \ell_P}$ be the pull-back of $V$ by this homomorphism.
\begin{thm}\label{thm:twist of simple module}
Let $I(P;\chi,J,V;Q)$ be a simple representation.
Then we have $I(P;\chi,J,V;Q)^\iota = I(P;\chi,S_{\aff,P,\chi}\setminus J,V_{\ell - \ell_p};Q^c)$ where $\Delta_{Q^c} = \Delta_P\cup (\Delta(\sigma)\setminus \Delta_Q)$.
\end{thm}
\begin{proof}
Since $I(P;\chi,J,V;Q) = I_{P(\sigma)}(\St^{P(\sigma)}_Q(\pi_{\chi,J,V}))$ is simple, by Corollary~\ref{cor:I=I' for simples}, we have $I(P;\chi,J,V;Q)^\iota = I_{P(\sigma)}((\St^{P(\sigma)}_Q(\pi_{\chi,J,V}))_{\ell - \ell_{P(\sigma)}}^{\iota_{P(\sigma)}})$.
By Theorem~\ref{thm:twist of Steinberg}, Proposition~\ref{prop:twist of supersingular} and \cite[Lemma~4.9]{arXiv:1612.01312}, we have
\begin{align*}
I(P;\chi,J,V;Q)^\iota & \simeq I_{P(\sigma)}(\St^{P(\sigma)}_{Q^c}((\pi_{\chi,J,V})_{\ell - \ell_{P}}^{\iota_{P}}))\\
& \simeq I_{P(\sigma)}(\St^{P(\sigma)}_{Q^c}((\pi_{\chi,S_{\aff,P,\chi}\setminus J,V})_{\ell - \ell_{P}})).
\end{align*}
Let $\Xi$ be a character of $\mathcal{H}_{\aff,P}$ defined by the pair $\chi$ and $S_{\aff,P,\chi}\setminus J$.
Put $\mathcal{H}_{P,\Xi} = \mathcal{H}_{\aff,P}C[\Omega_P(1)_\Xi]$.
Then $\pi_{\chi,S_{\aff,P,\chi}\setminus J,V} = (\Xi\otimes V)\otimes_{\mathcal{H}_{P,\Xi}}\mathcal{H}_P$.
Let $f\colon \mathcal{H}_P\to \mathcal{H}_P$ be an algebra homomorphism defined by $f(T_w^P) = (-1)^{\ell(w) - \ell_P(w)}T_w^P$.
Then $f$ preserves $\mathcal{H}_{\aff,P}$ and $C[\Omega_P(1)_\Xi]$ and we have $(\pi_{\chi,S_{\aff,P,\chi}\setminus J,V})_{\ell - \ell_P} = \pi_{\chi,S_{\aff,P,\chi}\setminus J,V}\circ f = ((\Xi\circ f)\otimes (V\circ f))\otimes_{\mathcal{H}_{P,\Xi}}\mathcal{H}_P$.
By the definition, $V\circ f = V_{\ell - \ell_P}$.
By \cite[Lemma~4.7]{arXiv:1612.01312}, $f$ is identity on $\mathcal{H}_{\aff,P}$.
Hence $\Xi\circ f = \Xi$.
Hence $(\pi_{\chi,S_{\aff,P,\chi}\setminus J,V})_{\ell - \ell_P} = (\Xi\otimes V_{\ell - \ell_P})\otimes_{\mathcal{H}_{P,\Xi}}\mathcal{H}$ and we get the theorem.
\end{proof}

\subsection{Structure of $I'_P$}
Assume that $C$ is an algebraically closed field.
\begin{prop}
Let $P$ be a parabolic subgroup and $\sigma$ a simple supersingular representation of $\mathcal{H}_P$.
Then for each parabolic subgroup $Q$ between $P$ and $P(\sigma)$, there exists a submodule $\pi_Q\subset I'_P(\sigma)$ such that 
\begin{enumerate}
\item if $Q_1\subset Q_2$ then $\pi_{Q_1}\subset \pi_{Q_2}$.
\item $\pi_Q/\sum_{Q_1\subsetneq Q}\pi_{Q_1} = I(P,\sigma,Q)$.
\end{enumerate}
\end{prop}
Compare with $I_Q(e_Q(\sigma))\subset I_P(\sigma)$.
In other words, the structure of $I'_P(\sigma)$ is ``opposite to'' that of $I_P(\sigma)$.
\begin{proof}
First assume that $P(\sigma) = G$.
Put $\sigma' = \sigma_{\ell -\ell_P}^{\iota_P}$.
Then $I'_P(\sigma) = I_P(\sigma')^{\iota}$.
Set $\pi_Q = I_{Q^c}(e_{Q^c}(\sigma'))^{\iota}\subset I'_P(\sigma)$ where $\Delta_{Q^c} = (\Delta\setminus\Delta_Q)\cup \Delta_P$.
Then the first condition is satisfied.
Since $Q_1\subset Q_2$ if and only if $Q_1^c\supset Q_2^c$, we have
\begin{align*}
\pi_Q/\sum_{Q_1\subsetneq Q}\pi_{Q_1}
& = \left(I_{Q^c}(e_{Q^c}(\sigma'))/\sum_{Q_1\subsetneq Q}I_{Q_1^c}(e_{Q_1^c}(\sigma'))\right)^\iota\\
& = \left(I_{Q^c}(e_{Q^c}(\sigma'))/\sum_{Q_1\supsetneq Q^c}I_{Q_1}(e_{Q_1}(\sigma'))\right)^\iota\\
& = (\St_{Q^c}(\sigma'))^\iota = \St_Q(\sigma)
\end{align*}
by Theorem~\ref{thm:twist of Steinberg}.
By the assumption $P(\sigma) = G$, we have $\St_Q(\sigma) = I(P,\sigma,Q)$.
We get the proposition in this case.

In general, applying the proposition for $I_P^{P(\sigma)\prime}(\sigma)$, we get $\pi'_Q\subset I_P^{P(\sigma)\prime}(\sigma)$ for each $P(\sigma)\supset Q\supset P$.
Put $\pi_Q = I'_{P(\sigma)}(\pi'_Q)$.
The first condition is obvious.
For the second condition, we have
\begin{align*}
\pi_Q/\sum_{Q_1\subsetneq Q}\pi_{Q_1}
& =
I'_{P(\sigma)}(\pi'_Q)/\sum_{Q_1\subsetneq Q}I'_{P(\sigma)}(\pi'_{Q_1})\\
& \simeq
I'_{P(\sigma)}\left(\pi'_Q/\sum_{Q_1\subsetneq Q}\pi'_{Q_1}\right)\\
& \simeq
I'_{P(\sigma)}(\St_Q^{P(\sigma)}(\sigma)).
\end{align*}
Since $I(P,\sigma,Q) = I_{P(\sigma)}(\St_Q^{P(\sigma)}(\sigma))$ is simple, by Corollary~\ref{cor:I=I' for simples}, we have $I(P,\sigma,Q)\simeq I'_{P(\sigma)}(\St_Q^{P(\sigma)}(\sigma))$.
Now we get the proposition.
\end{proof}

\section{Dual}\label{sec:Dual}
We have an anti-automorphism $\zeta = \zeta_G \colon \mathcal{H}\to \mathcal{H}$ defined by $\zeta(T_w) = T_{w^{-1}}$.
Hence for a representation $\pi$, its linear dual $\pi^* = \Hom_C(\pi,C)$ has a structure of a right $\mathcal{H}$-module defined by $(fX)(v) = f(v\zeta(X))$ for $f\in\pi^*, v\in\pi$ and $X\in \mathcal{H}$.
Since any simple representation is finite-dimensional, if $\pi$ is simple then $\pi^*$ is again simple.
In this section, we compute $\pi^*$.

\begin{lem}
We have $\zeta(T_w^*) = T_{w^{-1}}^*$.
\end{lem}
\begin{proof}
In $\mathcal{H}[q_s^{\pm 1}]$, we have $\zeta(T_w^*) = \zeta(q_wT_{w^{-1}}^{-1}) = q_wT_w^{-1} = T_{w^{-1}}^*$.
\end{proof}

\subsection{Parabolic inductions}
In this subsection, we calculate $I_P(\sigma)^*$.
Let $P' = n_{w_Gw_P}\opposite{P}n_{w_Gw_P}^{-1}$.
Then we have $I_P(\sigma) \simeq n_{w_Gw_P}\sigma\otimes_{(\mathcal{H}_P^+,j_P^+)}\mathcal{H}$ by \cite[Proposition~2.21]{arXiv:1612.01312}.
Hence we have
\begin{align*}
\Hom_C(I_P(\sigma),C)
& \simeq \Hom_C(n_{w_Gw_P}\sigma\otimes_{(\mathcal{H}_{P'}^+,j_{P'}^+)}\mathcal{H},C)\\
& \simeq \Hom_{(\mathcal{H}_{P'}^+,j_{P'}^+)}(\mathcal{H},\Hom_C(n_{w_Gw_P}\sigma,C)).
\end{align*}
Therefore $I_P(\sigma)^* \simeq \Hom_{(\mathcal{H}_{P'}^+,j_{P'}^+)}(\mathcal{H},\Hom_C(n_{w_Gw_P}\sigma,C))$ here the action on the right hand side is twisted by $\zeta$.
Let $\varphi\in \Hom_{(\mathcal{H}_{P'}^+,j_{P'}^+)}(\mathcal{H},\Hom_C(n_{w_Gw_P}\sigma,C))$ and set $\varphi^\zeta = \varphi\circ\zeta$.
Let $w\in W_{P'}(1)$ which is $P'$-negative.
Then $w^{-1}$ is $P'$-positive.
Hence for $X\in \mathcal{H}$ and $x\in n_{w_Gw_P}\sigma$, we have $\varphi^\zeta(X T_w)(x) = \varphi(\zeta(XT_{w}))(x) = \varphi(T_{w^{-1}}\zeta(X))(x) = (T^{P'}_{w^{-1}}\varphi(\zeta(X)))(x) = \varphi(\zeta(X))(xT^{P'}_{w^{-1}}) = \varphi^\zeta(X)(x\zeta_P(T^{P'}_w))$.
(Here we regard $\Hom_C(n_{w_Gw_P}\sigma,C)$ as a left $\mathcal{H}_{P'}$-module.)
Therefore 
\[
\varphi^\zeta \in \Hom_{(\mathcal{H}_{P'}^-,j_{P'}^-)}(\mathcal{H},(n_{w_Gw_P}\sigma)^*).
\]
For $X,Y\in \mathcal{H}$, we have $(\varphi Y)^\zeta(X) = (\varphi Y)(\zeta (X)) = \varphi(\zeta(X)\zeta(Y)) = \varphi(\zeta(YX)) = \varphi^\zeta(YX) = (\varphi^\zeta Y)(X)$.
Hence $\varphi\mapsto \varphi^\zeta$ induces
\[
I_P(\sigma)^*\simeq \Hom_{(\mathcal{H}_{P'}^-,j_{P'}^-)}(\mathcal{H},(n_{w_Gw_P}\sigma)^*) = I'_{P'}(n_{w_Gw_P}\sigma^*).
\]
\begin{prop}\label{prop:dual of parabolic induction}
We have $I_P(\sigma)^*\simeq I'_{P'}(n_{w_Gw_P}\sigma^*)$.
\end{prop}
The same calculation shows that 
\begin{prop}\label{prop:dual of the other parabolic induction}
We have $I'_P(\sigma)^*\simeq I_{P'}(n_{w_Gw_P}\sigma^*)$.
\end{prop}
\begin{rem}
These propositions are true for any commutative ring $C$.
\end{rem}

\subsection{Steinberg modules}
Let $P$ be a parabolic subgroup, $\sigma$ an $\mathcal{H}_P$-module such that $P(\sigma) = G$ and $P_2$ the parabolic subgroup corresponding to $\Delta\setminus\Delta_P$.
We calculate $(\St_Q\sigma)^{*}$.
\begin{prop}\label{prop:dual of Steinberg modules}
Let $Q$ be a parabolic subgroup containing $P$ and put $Q' = n_{w_Gw_Q}\opposite{Q}n_{w_Gw_Q}^{-1}$.
Then $(\St_Q\sigma)^*\simeq \St_{Q'}\sigma^*$.
\end{prop}

We start with the case of $Q = G$.
\begin{lem}\label{lem:dual of extension}
We have $e_G(\sigma)^*\simeq e_G(\sigma^*)$.
\end{lem}
\begin{proof}
Let $f\in e_G(\sigma)^*$ and $x\in e_G(\sigma)$.
For $w\in W_P(1)$, we have 
\begin{align*}
(fe_G(\sigma)^*(T^*_w))(x) & = f(xe_G(\sigma)(\zeta(T^*_w))) = f(xe_G(\sigma)(T^*_{w^{-1}}))\\
& = f(x\sigma(T_{w^{-1}}^{P*})) = f(x\sigma(\zeta_P(T_{w}^{P*}))) = (f\sigma^*(T_w^{P*}))(x).
\end{align*}
Hence $e_G(\sigma)^*(T_w^*) = \sigma^*(T_w^{P*})$.
For $w\in W_{\aff,P_2}(1)$, we have $(fe_G(\sigma)^*(T^*_w))(x) = f(xe_G(\sigma)(\zeta(T^*_w))) = f(xe_G(\sigma)(T^*_{w^{-1}})) = f(x)$.
Hence $e_G(\sigma)^*(T^*_w) = 1$.
Therefore by the characterization of $e_G(\sigma^*)$, we have the lemma.
\end{proof}

\begin{proof}[Proof of Proposition~\ref{prop:dual of Steinberg modules}]
By Lemma~\ref{lem:exact sequence related to I and I'}, we have the following exact sequence
\[
0\to \St_Q(\sigma)\to I'_Q(e_Q(\sigma))\to \bigoplus_{Q_1\supsetneq Q}I'_{Q_1}(e_{Q_1}(\sigma)).
\]
Taking dual, we get an exact sequence
\[
\bigoplus_{Q_1\supsetneq Q}I'_{Q_1}(e_{Q_1}(\sigma))^*\to I'_Q(e_Q(\sigma))^*\to \St_Q(\sigma)^*\to 0.
\]
Put $Q' = n_{w_Gw_Q}\opposite{Q}n_{w_Gw_Q}^{-1}$.
Then by Proposition~\ref{prop:dual of the other parabolic induction}, we have $I'_Q(e_Q(\sigma))^* = I_{Q'}(n_{w_Gw_Q}e_Q(\sigma)^*) = I_{Q'}(e_{Q'}(\sigma)^*) = I_{Q'}(e_{Q'}(\sigma^*))$ by Lemma~\ref{lem:dual of extension} and \cite[Lemma~2.27]{arXiv:1612.01312}.
Put $Q'_1 = n_{w_Gw_{Q_1}}\opposite{Q}_1n_{w_Gw_{Q_1}}^{-1}$.
Then 
\[
\bigoplus_{Q_1\supsetneq Q}I_{Q'_1}(e_{Q'_1}(\sigma^*))\to I_{Q'}(e_{Q'}(\sigma^*))\to \St_Q(\sigma)^*\to 0.
\]
Since $\Delta_{Q'} = -w_G(\Delta_Q)$ and $Q'_1 = -w_G(\Delta_{Q_1})$, we have $Q_1\supsetneq Q$ if and only if $Q'_1\supsetneq Q'$.
Hence 
\[
\bigoplus_{Q_2\supsetneq Q'}I_{Q_2}(e_{Q_2}(\sigma^*))\to I_{Q'}(e_{Q'}(\sigma^*))\to \St_Q(\sigma)^*\to 0.
\]
By the lemma below, we get $\St_Q(\sigma)^* = \St_{Q'}(\sigma^*)$.
\end{proof}
\begin{lem}
Let $Q_1\supset Q$ be parabolic subgroups.
Put $Q' = n_{w_Gw_Q}\opposite{Q}n_{w_Gw_Q}^{-1}$ and $Q'_1 = n_{w_Gw_{Q_1}}\opposite{Q_1}n_{w_Gw_{Q_1}}^{-1}$.
Then the homomorphism induced by $I'_Q(e_Q(\sigma))\to I'_{Q_1}(e_{Q_1}(\sigma))$ with the dual is the inclusion $I_{Q_1'}(e_{Q_1'}(\sigma^*))\hookrightarrow I_{Q'}(e_{Q'}(\sigma^*))$ times $(-1)^{\ell(w_{Q_1}w_Q)}$.
\end{lem}
\begin{proof}
By Proposition~\ref{prop:explicit formula of I'_Q(e_Q) -> I'_(Q_1)(e_(Q_1))}, the homomorphism $I'_Q(e_Q(\sigma))\to I'_{Q_1}(e_{Q_1}(\sigma))$ is given by $\varphi\mapsto (X\mapsto (-1)^{\ell(w_{Q_1}w_Q)}\varphi(XT_{n_{w_{Q_1}w_Q}}^*))$.
We recall that the isomorphism $I'_Q(e_Q(\sigma))\simeq e_{Q'}(\sigma)\otimes_{(\mathcal{H}_{Q'}^+,j_{Q'}^{+})}\mathcal{H}$ is given by $\varphi\mapsto \sum_{w\in {}^{Q'}W_0}\varphi(T_{n_{w^{-1}w_Gw_Q}}^*)\otimes T_{n_{w}}$ \cite[Lemma~2.22]{arXiv:1612.01312}.
Let $\varphi'\in I'_{Q_1}(e_{Q_1}(\sigma))$ be the image of $\varphi$.
Then the image of $\varphi'$ in $e_{Q'_1}(\sigma)\otimes_{(\mathcal{H}^+_{Q'_1},j_{Q'_1}^{+})}\mathcal{H}$ is 
\begin{align*}
&\sum_{w\in {}^{Q'_1}W_0}\varphi'(T_{n_{w^{-1}w_Gw_{Q_1}}}^*)\otimes T_{n_w}\\
& = 
(-1)^{\ell(w_{Q_1}w_Q)}\sum_{w\in {}^{Q'_1}W_0}\varphi(T_{n_{w^{-1}w_Gw_{Q_1}}}^*T_{n_{w_{Q_1}w_Q}}^*)\otimes T_{n_w}.
\end{align*}
Since $w_{Q_1}w_Q\in W_{0,Q_1}$ and $w^{-1}w_Gw_{Q_1}\in W^{Q_1}_0$ \cite[Lemma~2.22]{arXiv:1612.01312}, we have $\ell(w^{-1}w_Gw_{Q_1}) + \ell(w_{Q_1}w_Q) = \ell(w^{-1}w_Gw_Q)$.
Hence we have $T_{n_{w^{-1}w_Gw_{Q_1}}}^*T_{n_{w_{Q_1}w_Q}}^* = T_{n_{w^{-1}w_Gw_{Q}}}^*$.
Therefore
\[
\sum_{w\in {}^{Q'_1}W_0}\varphi'(T_{n_{w^{-1}w_Gw_{Q_1}}}^*)\otimes T_{n_w} = 
(-1)^{\ell(w_{Q_1}w_Q)}\sum_{w\in {}^{Q'_1}W_0}\varphi(T_{n_{w^{-1}w_Gw_{Q}}}^*)\otimes T_{n_w}.
\]
Let $P_2$ be a parabolic subgroup corresponding to $\Delta\setminus\Delta_P$.
Let $w\in {}^{Q'}W_0$ but $w\notin {}^{Q'_1}W_0$.
Then there exists a simple reflection $s\in W_{0,Q'_1}$ such that $sw < w$.
Since $w\in {}^{Q'}W_0\subset W_{0,P_2}$, we have $s\in S_{0,P_2}$.
Hence for any $x\in e_{Q'_1}(\sigma)$, we have $x\otimes T_{n_w} = x\otimes T_{n_s}T_{n_{sw}} = xe_{Q'_1}(\sigma)(T^{Q'_1}_{n_s})\otimes T_{n_{sw}} = 0$ since $e_{Q'_1}(\sigma)(T^{Q'_1}_{n_s}) = 0$.
Hence
\begin{align*}
&(-1)^{\ell(w_{Q_1}w_Q)}\sum_{w\in {}^{Q'_1}W_0}\varphi(T_{n_{w^{-1}w_Gw_{Q}}}^*)\otimes T_{n_w}\\
&=
(-1)^{\ell(w_{Q_1}w_Q)}\sum_{w\in {}^{Q'}W_0}\varphi(T_{n_{w^{-1}w_Gw_{Q}}}^*)\otimes T_{n_w}.
\end{align*}
Therefore the homomorphism 
\[
e_{Q'}(\sigma)\otimes_{(\mathcal{H}_{Q'}^+,j_{Q'}^{+})}\mathcal{H}\to e_{Q'_1}(\sigma)\otimes_{(\mathcal{H}_{Q_1'}^+,j_{Q_1'}^{+})}\mathcal{H}
\]
is given by $x\otimes X\mapsto (-1)^{\ell(w_{Q_1}w_Q)}x\otimes X$.
(Here we identify $x\in e_{Q'}(\sigma)$ with $x\in e_{Q'_1}(\sigma)$.)

The isomorphism 
\[
(e_{Q'}(\sigma)\otimes_{(\mathcal{H}_{Q'}^+,j_{Q'}^{+*})}\mathcal{H})^*\simeq \Hom_{(\mathcal{H}_{Q'}^-,j_{Q'}^{-*})}(\mathcal{H},e_Q(\sigma)^*) = I_Q(e_Q(\sigma)^*)
\]
is given by $f\mapsto (X\mapsto (x\mapsto f(x\otimes \zeta(X))))$ and the opposite is given by $f'\mapsto ((x\otimes X)\mapsto f'(\zeta(X))(x))$. (Here we identify $e_Q(\sigma)^*$ with $e_Q(\sigma^*)$.)
Hence the maps
\[
I_{Q_1}(e_{Q_1}(\sigma)^*)\simeq (e_{Q'_1}(\sigma)\otimes_{(\mathcal{H}_{Q'_1}^+,j_{Q'_1}^{+})}\mathcal{H})^*
\to
(e_{Q'}(\sigma)\otimes_{(\mathcal{H}_{Q'}^+,j_{Q'}^{+})}\mathcal{H})^* \simeq I_{Q}(e_{Q}(\sigma)^*)
\]
send $f\in I'_{Q_1}(e_{Q_1}(\sigma)^*)$ to
\begin{align*}
(x\otimes X) \mapsto f(\zeta(X))(x)& \in (e_{Q'_1}(\sigma)\otimes_{(\mathcal{H}_{Q'_1}^+,j_{Q'_1}^{+})}\mathcal{H})^*,\\
(x\otimes X) \mapsto (-1)^{\ell(w_{Q_1}w_Q)}f(\zeta(X))(x)& \in (e_{Q'}(\sigma)\otimes_{(\mathcal{H}_{Q'}^+,j_{Q'}^{+})}\mathcal{H})^*,\\
X \mapsto (x\mapsto (-1)^{\ell(w_{Q_1}w_Q)}f(X)(x))& \in I_{Q}(e_{Q}(\sigma)^*).
\end{align*}
Namely, it is equal to the the natural embedding times $(-1)^{\ell(w_{Q_1}w_Q)}$.
\end{proof}

\subsection{Supersingular modules}
Assume that $C$ is a field.
\begin{thm}\label{thm:dual of Supersingular representation}
Let $(\chi,J,V)$ be as in subsection~\ref{subsec:supersingulars}.
Then we have $\pi_{\chi,J,V}^* \simeq \pi_{\chi^{-1},J,V^*}$.
\end{thm}
\begin{proof}
Let $\Xi$ be a character of $\mathcal{H}_{\aff}$ determined by $(\chi,J)$.
By the proof of \cite[Proposition~6.17]{Vigneras-prop-III}, $\Xi\otimes V\subset \pi_{\chi,J,V}|_{\mathcal{H}_\Xi}$ is a direct summand.
Hence $(\Xi\otimes V)^*\subset (\pi_{\chi,J,V})^*|_{\mathcal{H}_\Xi}$.
Since $\Xi$ and $V$ are finite-dimensional, we have $(\Xi\otimes V)^* = \Xi^*\otimes V^*$.
Therefore we have a non-zero homomorphism $(\Xi^*\otimes V^*)\otimes_{\mathcal{H}_\Xi}\mathcal{H}\to \pi_{\chi,J,V}^*$.
The restriction of $V^*$ to $Z_\kappa\cap W_{\aff,P}(1)$ is the direct sum of $\chi^* = \chi^{-1}$ since $V|_{Z_\kappa\cap W_{\aff,P}(1)}$ is a direct sum of $\chi$.
For $s\in S_{\aff,\chi} = S_{\aff,\chi^{-1}}$, $\Xi^*(T_{\widetilde{s}}) = \Xi(\zeta(T_{\widetilde{s}})) = \Xi(T_{\widetilde{s}^{-1}})$ where $\widetilde{s}$ is a lift of $s$.
This is $0$ or $\chi(c_{\widetilde{s}^{-1}})$ and $0$ if and only if $s\in J$.
Hence the subset of $S_{\aff,\chi^{-1}} = S_{\aff,\chi}$ attached to $\Xi^*$ is $J$.
Therefore $(\Xi^*\otimes V^*)\otimes_{\mathcal{H}_\Xi}\mathcal{H} = \pi_{\chi^{-1},J,V^*}$.
Hence we get a non-zero homomorphism $\pi_{\chi^{-1},J,V^*}\to \pi_{\chi,J,V}^*$.
Since this is a non-zero homomorphism between simple modules, this is an isomorphism.
\end{proof}

\subsection{Simple modules}
Assume that $C$ is a field.
Combining Proposition~\ref{prop:dual of parabolic induction}, \ref{prop:dual of Steinberg modules} and Theorem~\ref{thm:dual of Supersingular representation}, we get the following theorem.
\begin{thm}\label{thm:dual of simples}
Set $P' = n_{w_Gw_P}\opposite{P}n_{w_Gw_P}^{-1}$ and $Q' = n_{w_Gw_Q}\opposite{Q}n_{w_Gw_Q}^{-1}$.
Let $(\chi',J',V')$ be a triple for $\mathcal{H}_{P'}$ defined by the pull-back of the triple $(\chi^{-1},J,V^*)$ by $n_{w_Gw_P}$.
Then we have $I(P;\chi,J,V;Q)^* = I(P';\chi',J',V';Q')$.
\end{thm}

We use the following lemma.
\begin{lem}
Let $P$ be a parabolic subgroup and $\sigma$ an $\mathcal{H}_P$-module.
Then we have $P(n_{w_Gw_{P}}\sigma) = n_{w_Gw_{P(\sigma)}}\opposite{P(\sigma)}n_{w_Gw_P(\sigma)}^{-1}$.
\end{lem}
\begin{rem}
By \cite[Lemma~2.27]{arXiv:1612.01312}, we have $n_{w_Gw_{P(\sigma)}}\sigma = n_{w_Gw_P}\sigma$.
\end{rem}
\begin{proof}
Let $P_\alpha$ be the parabolic subgroup corresponding to $\Delta_P\cup \{\alpha\}$ where $\alpha\in\Delta\setminus\Delta_P$.
Set $n = n_{w_Gw_P}$ and $P' = n_{w_Gw_P}\opposite{P}n_{w_Gw_P}^{-1}$.
Let $\alpha\in\Delta(n\sigma)\setminus\Delta_{P'}$ and we prove that $\alpha\in-w_G(\Delta(\sigma)\setminus\Delta_P)$.
Since $\Delta_{P'} = -w_G(\Delta_P)$, $\langle \alpha,\Delta_{P'}\rangle = 0$ implies $\langle -w_G(\alpha),\Delta_P\rangle = 0$.
If for any $\lambda\in \Lambda(1)\cap W_{\aff,P_\alpha}(1)$ satisfies $(n\sigma)(T_\lambda^{P'}) = 1$, then we have $\sigma(T_{n^{-1}\lambda n}^P) = 1$.
Hence $\sigma(T_\lambda^P) = 1$ for any $\lambda\in n^{-1}(\Lambda(1)\cap W_{\aff,P_\alpha}(1))n = \Lambda(1)\cap W_{\aff,P_{-(w_Gw_P)^{-1}(\alpha)}}(1)$.
Since $\langle -w_G(\alpha),\Delta_P\rangle = 0$, we have $-(w_Gw_P)^{-1}(\alpha) = -w_Pw_G(\alpha) = -w_G(\alpha)$.
Hence $-w_G(\alpha)\in \Delta(\sigma)$.
\end{proof}

\begin{proof}[Proof of Theorem~\ref{thm:dual of simples}]
Set $\sigma = \pi_{\chi,J,V}$ and $P(\sigma)' = n_{w_Gw_{P(\sigma)}}\opposite{P(\sigma)}n_{w_Gw_{P(\sigma)}}^{-1}$.
We have
\begin{align*}
I(P;\chi,J,V;Q)^* & = I_{P(\sigma)}(\St_Q^{P(\sigma)}\sigma)^*\\
& \simeq I_{P(\sigma)'}(n_{w_Gw_{P(\sigma)}}(\St_Q^{P(\sigma)}\sigma)^*)\\
& \simeq I_{P(\sigma)'}(n_{w_Gw_{P(\sigma)}}(\St_{n_{w_{P(\sigma)w_Q}}\opposite{Q}n_{w_{P(\sigma)w_Q}}^{-1}}^{P(\sigma)}\sigma^*)).
\end{align*}
The adjoint action of $n_{w_Gw_{P(\sigma)}}$ induces an isomorphism $\mathcal{H}_{P(\sigma)}\simeq\mathcal{H}_{P(\sigma)'}$.
For a parabolic subgroup $Q_1$ between $P(\sigma)$ and $P$, let $Q_2$ be a parabolic subgroup corresponding to $w_Gw_{P(\sigma)}(\Delta_{Q_1})$.
Then the adjoint action of $n_{w_Gw_{P(\sigma)}}$ induces an isomorphism $\mathcal{H}_{Q_1}\simeq \mathcal{H}_{Q_2}$ and sends $\mathcal{H}_{Q_1}^{P(\sigma)-}$ to $\mathcal{H}_{Q_2}^{P(\sigma)'-}$.
Moreover, it is compatible with homomorphisms $j_{Q_1}^{P(\sigma)-*}$ and $j_{Q_2}^{P(\sigma)'-*}$.
Hence we have $n_{w_Gw_{P(\sigma)}}I_{Q_1}^{P(\sigma)}(e_{Q_1}(\sigma))\simeq I_{Q_2}^{P(\sigma)'}(n_{w_Gw_{P(\sigma)}}e_{Q_1}(\sigma))$.

Since $\Delta_{P(\sigma)}\setminus\Delta_{P}$ is orthogonal to $\Delta_P$, $w_{P(\sigma)}w_P(\Delta_P) = \Delta_P$.
Hence we have $w_Gw_{P(\sigma)}(\Delta_P) = w_Gw_P(\Delta_P) = \Delta_{P'}$.
Therefore, from the above compatibility of negative algebras and homomorphisms, $n_{w_Gw_{P(\sigma)}}e_{Q_1}(\sigma) = e_{Q_2}(n_{w_Gw_{P(\sigma)}}\sigma)$ where $e_{Q_2}$ is the extension from $P'$ to $Q_2$.
Therefore, combining the formula in the above paragraph, we get $n_{w_Gw_{P(\sigma)}}\St_{Q_1}^{P(\sigma)}(\sigma)\simeq \St_{Q_2}^{P(\sigma)'}(n_{w_Gw_{P(\sigma)}}\sigma)$.

Now set $Q_1 = n_{w_{P(\sigma)w_Q}}\opposite{Q}n_{w_{P(\sigma)w_Q}}$.
Then $\Delta_{Q_1} = w_{P(\sigma)}w_Q(\Delta_Q)$.
Hence $\Delta_{Q_2} = w_Gw_Q(\Delta_Q) = \Delta_{Q'}$.
Therefore we have
\[
I(P;\chi,J,V;Q)^*
\simeq
I_{P(\sigma)'}(\St_{Q'}^{P(\sigma)'}n_{w_Gw_{P(\sigma)}}\sigma^*).
\]
By \cite[Lemma~2.27]{arXiv:1612.01312}, $\sigma^* = n_{w_P(\sigma)w_P}\sigma^*$.
Hence we get
\[
I(P,\sigma,Q)^*
\simeq
I_{P(\sigma)'}(\St_{Q'}^{P(\sigma)'}n_{w_Gw_P}\sigma^*).
\]

\end{proof}

\end{document}